\documentclass[a4paper]{amsart}
\usepackage{graphicx}
\usepackage{amsmath}
\usepackage{amssymb}
\usepackage{amsfonts}
\usepackage{amsthm}
\usepackage{eucal}
\usepackage{tikz}
\usetikzlibrary{cd}

\theoremstyle{plain}
\newtheorem{theorem}{Theorem}[section]

\newtheorem{lemma}[theorem]{Lemma}
\newtheorem{proposition}[theorem]{Proposition}

\newtheorem{corollary}[theorem]{Corollary}

\newtheorem*{proposition*}{Proposition}
\newtheorem*{corollary*}{Corollary}

\theoremstyle{definition}
\newtheorem{definition}[theorem]{Definition}
\newtheorem{example}[theorem]{Example}

\newtheorem{remark}[theorem]{Remark}

\makeatletter
\@addtoreset{theorem}{section}
\@addtoreset{subsection}{section}
\makeatother

\begin{document}

\title{Lie algebra models for unstable homotopy theory}

\author[Gijs Heuts]{Gijs Heuts}
\address{Utrecht University, Department of Mathematical Sciences, Budapestlaan 6, 3584CD Utrecht, the Netherlands}
\email{g.s.k.s.heuts@uu.nl}

\date{}

\begin{abstract}
Quillen showed how to describe the homotopy theory of simply-connected rational spaces in terms of differential graded Lie algebras. Here we survey a generalization of Quillen's results that describes the $v_n$-periodic localizations of homotopy theory (where rational corresponds to $n=0$) in terms of spectral Lie algebras. The latter form an extension of the theory of Lie algebras to the setting of stable homotopy theory. This is a chapter written for the Handbook of Homotopy Theory edited by Haynes Miller.
\end{abstract}

\maketitle

\tableofcontents

\section{Introduction}
\label{sec:introduction}

Write $\mathrm{Ho}(\mathcal{S}_*)$ for the homotopy category of pointed spaces. The tools of algebraic topology often amount to studying spaces by applying a variety of functors
\begin{equation*}
F: \mathrm{Ho}(\mathcal{S}_*) \rightarrow \mathrm{Ho}(\mathcal{A}),
\end{equation*}
for $\mathcal{A}$ a homotopy theory which is `algebraic' in nature. A typical example of such an $F$ is the functor $C^*(-: R)$ taking cochains with values in a commutative ring $R$. In this case $\mathcal{A}$ can be taken to be (the opposite of) the category of differential graded algebras over $R$, using the cup product of cochains as multiplication. If one wishes to take the rather subtle commutativity properties of the cup product into account, a more refined choice for $\mathcal{A}$ would be the category of $\mathbf{E}_\infty$-algebras over $R$.

An optimist might hope to choose $\mathcal{A}$ cleverly enough so that one can construct a functor $F$ which is an equivalence of categories. If one restricts the domain of $F$ to the homotopy category $\mathrm{Ho}(\mathcal{S}^{\geq 2}_{\mathbb{Q}})$ of simply-connected rational pointed spaces, then there is the following landmark result \cite{quillen}:

\begin{theorem}[Quillen \cite{quillen}]
\label{thm:rationalhomotopy}
There are equivalences of categories
\[
\begin{tikzcd}
& \mathrm{Ho}(\mathcal{S}^{\geq 2}_{\mathbb{Q}}) \ar{dl}[swap]{L_\mathbb{Q}} \ar{dr}{C_\mathbb{Q}} & \\
\mathrm{Ho}(\mathrm{Lie}(\mathrm{Ch}_{\mathbb{Q}})^{\geq 1}) & & \mathrm{Ho}(\mathrm{coCAlg}(\mathrm{Ch}_{\mathbb{Q}})^{\geq 2}),
\end{tikzcd}
\]
where $\mathrm{Lie}(\mathrm{Ch}_{\mathbb{Q}})^{\geq 1}$ denotes the category of connected differential graded Lie algebras\index{differential graded Lie algebra} over $\mathbb{Q}$ and $\mathrm{coCAlg}(\mathrm{Ch}_{\mathbb{Q}})^{\geq 2}$ the category of simply-connected differential graded cocommutative coalgebras\index{differential graded coalgebra} over $\mathbb{Q}$.
\end{theorem}

The functor $C_{\mathbb{Q}}$ is essentially a version of the rational chains; in particular, the homology of the underlying chain complex of $C_{\mathbb{Q}}(X)$ is the rational homology of the space $X$ and the coproduct is the usual one arising from the diagonal map of $X$. The functor $L_{\mathbb{Q}}$ is perhaps more surprising. It refines the collection of rational homotopy groups of $X$, in the sense that there is a natural isomorphism
\begin{equation*}
H_*(L_{\mathbb{Q}}(X)) \cong (\pi_{*+1} X) \otimes \mathbb{Q}.
\end{equation*}
The induced Lie bracket on $H_*(L_{\mathbb{Q}}(X))$ corresponds to the classical Whitehead product\index{Whitehead product} on the (rational) homotopy groups of $X$.

Theorem \ref{thm:rationalhomotopy} is a very satisfying result in itself. In conjunction with Sullivan's approach to rational homotopy theory \index{rational homotopy theory} using minimal models \cite{sullivan} it has led to many interesting developments within homotopy theory, as well as striking applications to geometry (see \cite{felixhalperinthomas, felixopreatanre} for an overview). Moreover, it raises a number of natural questions. Are there subcategories of $\mathrm{Ho}(\mathcal{S}_*)$ other than that of rational spaces which can be described in terms of Lie algebras and/or coalgebras? Even better, does the entire category $\mathrm{Ho}(\mathcal{S}_*)$ admit such an algebraic model? This survey mostly concerns an answer to the first question, although we include some discussion of the second one at the very end of Section \ref{sec:questions}.

Rationalization is the first step in a hierarchy of localizations of $\mathcal{S}_*$, which we will refer to as the \emph{$v_n$-periodic localizations}\index{$v_n$-periodic localization}. For each prime $p$ there is a family of these, indexed by the natural numbers $n \geq 0$. The case $n=0$ is precisely rational homotopy theory. The $v_n$-periodic localizations are in a precise sense the most elementary (or `prime') localizations of homotopy theory. Our main objective will be to discuss the following generalization of Quillen's Theorem \ref{thm:rationalhomotopy} (see \cite{heuts}):

\begin{theorem}
\label{thm:vnperiodicspaces}
There is an equivalence of homotopy theories
\begin{equation*}
L_{v_n}\colon \mathcal{S}_{v_n} \longrightarrow \mathrm{Lie}(\mathrm{Sp}_{v_n}),
\end{equation*}
with $\mathcal{S}_{v_n}$ denoting the $v_n$-periodic localization of $\mathcal{S}_*$ and $\mathrm{Lie}(\mathrm{Sp}_{v_n})$ the $v_n$-periodic localization of the homotopy theory of \emph{spectral Lie algebras}. \index{spectral Lie algebra}
\end{theorem} 

Explaining this result will of course require a discussion of spectral Lie algebras, a concept that has only recently surfaced. It has several promising applications within homotopy theory and interesting connections to derived algebraic geometry. 

We now give an overview of the contents of this chapter. In Section \ref{sec:rationalhomotopy} we review Quillen's rational homotopy theory\index{rational homotopy theory}. In particular, we will highlight the relation between the Lie algebra model $L_{\mathbb{Q}}(X)$ and the coalgebra model $C_{\mathbb{Q}}(X)$, which can be expressed in terms of Koszul duality\index{Koszul duality}. Since it plays an important role in this survey, we devote Section \ref{sec:koszul} to a general discussion of this duality. It is preceded by Section \ref{sec:monads}, which gives a short summary of monads and their algebras. In Section \ref{sec:Lie}  we discuss spectral Lie algebras and their connection with the Goodwillie tower of the identity functor \index{Goodwillie tower of the identity} of $\mathcal{S}_*$. In Section \ref{sec:periodicity} we turn to periodic phenomena in homotopy theory. We describe the general philosophy of stable chromatic homotopy theory\index{chromatic homotopy theory}, in analogy with the theory of localizations in algebra, and its application to unstable homotopy theory through the Bousfield--Kuhn functors\index{Bousfield--Kuhn functor}. In Section \ref{sec:Liealgebras} we come to Theorem \ref{thm:vnperiodicspaces}. The reader might note the absence of coalgebras in its statement. There does exist a coalgebra model for $\mathcal{S}_{v_n}$, but unlike the rational case it turns out that the Lie algebra model plays a preferred role for $n > 0$. Loosely speaking, coalgebras can only be used to describe a certain \emph{completion} of $\mathcal{S}_{v_n}$, rather than this homotopy theory itself. We will discuss this point in Section \ref{sec:Liealgebras} as well. In the final Section \ref{sec:questions} we discuss examples and several open questions which deserve further investigation.

We stress that the style of this survey is rather informal. To counter this we include many pointers to the literature, where most of our statements and their proofs can be found with a more technically precise treatment. Still, we provide sketches of the proofs of most major results (sometimes from a non-standard perspective), hoping to make this chapter self-contained enough to convey to the reader all of the essential ideas involved. The reader might also be interested in the survey of Behrens--Rezk \cite{behrensrezksurvey}, which concerns some of the same topics we describe here.

\emph{Conventions and notation.} Thus far we have not been specific about the formalism for homotopy theory we use. Most statements in this survey are sufficiently generic for the reader to adapt to their preferred setting. However, for convenience we work in the setting of $\infty$-categories \cite{joyal,htt}. This handbook also includes an expository account by Groth. Thus $\mathcal{S}_*$ and $\mathrm{Sp}$ will denote the $\infty$-categories of pointed spaces and of spectra respectively. We write $\mathcal{S}_*^{\geq n}$ for the full subcategory of $\mathcal{S}_*$ on the $(n-1)$-connected spaces. We will use the standard notation
\[
\begin{tikzcd}
\mathcal{S}_* \ar[shift left]{r}{\Sigma^\infty} & \mathrm{Sp} \ar[shift left]{l}{\Omega^\infty}
\end{tikzcd}
\] 
for the adjunction between the suspension spectrum $\Sigma^\infty$ and the zeroth space $\Omega^\infty$. All limits and colimits are to be thought of as $\infty$-categorical ones (unless explicitly stated otherwise); in other words, they are \emph{homotopy} limits and colimits. For $\infty$-categories $\mathcal{C}$ and $\mathcal{D}$ we denote by $\mathrm{Fun}(\mathcal{C},\mathcal{D})$ the $\infty$-category of functors between them. Our reference for algebra in the setting of $\infty$-categories is \cite{ha}. When discussing coalgebras we will often use the adjective \emph{commutative} (rather than cocommutative) when no confusion can arise, i.e., when we are not simultaneously considering an algebra structure.

\section{Rational homotopy theory}
\label{sec:rationalhomotopy}

The results of Quillen and Sullivan on rational homotopy theory have been extensively documented \cite{quillen,sullivan,bousfieldgugenheim,felixhalperinthomas,lurierational}; we will therefore feel free to give a somewhat non-traditional exposition with an eye towards the results of later sections. 
\index{rational homotopy theory}

A map $f\colon X \rightarrow Y$ of simply-connected pointed spaces is a \emph{rational equivalence} if it induces an isomorphism on rational homotopy groups, i.e., if for every $n \geq 2$ the map
\begin{equation*}
\pi_n f \otimes \mathbb{Q}\colon \pi_n X \otimes \mathbb{Q} \rightarrow \pi_n Y \otimes \mathbb{Q}
\end{equation*}
is an isomorphism. Equivalently, $f$ induces an isomorphism on rational homology. A simply-connected pointed space $X$ is said to be \emph{rational} if for each $n \geq 2$ the abelian group $\pi_n X$ is uniquely divisible, i.e., is a vector space over $\mathbb{Q}$. Equivalently, each homology group $H_n(X; \mathbb{Z})$ is a rational vector space (with $n \geq 2$ again). We write $\mathcal{S}_{\mathbb{Q}}^{\geq 2}$ for the full subcategory of $\mathcal{S}_*^{\geq 2}$ on the rational spaces.

Every simply-connected space admits a \emph{rationalization} $\eta_X: X \rightarrow X_{\mathbb{Q}}$, which is a map satisfying the following two properties:
\begin{itemize}
\item[(1)] The space $X_{\mathbb{Q}}$ is rational.
\item[(2)] The map $\eta_X$ is a rational equivalence.
\end{itemize}
Moreover, the assignment $X \mapsto X_{\mathbb{Q}}$ can be made functorial and $\eta_X$ can be made a natural transformation. The following terminology is convenient to describe the situation: 

\begin{definition}
\label{def:localization}
\index{localization}
Suppose $\mathcal{C}$ is an $\infty$-category and $W$ a class of morphisms in $\mathcal{C}$. A functor $L\colon \mathcal{C} \rightarrow \mathcal{D}$ is called a \emph{localization of $\mathcal{C}$ at $W$} if $L$ sends each morphism in $W$ to an equivalence in $\mathcal{D}$ and is universal with this property. More precisely, if $\mathcal{E}$ is any other $\infty$-category then precomposition by $L$ determines an equivalence of $\infty$-categories
\begin{equation*}
L^*: \mathrm{Fun}(\mathcal{D}, \mathcal{E}) \rightarrow \mathrm{Fun}_W(\mathcal{C},\mathcal{E}),
\end{equation*}
where the codomain denotes the full subcategory of $\mathrm{Fun}(\mathcal{C},\mathcal{E})$ of functors sending elements of $W$ to equivalences. A localization $L\colon \mathcal{C} \rightarrow \mathcal{D}$ is called \emph{reflective} if $L$ admits a fully faithful right adjoint. \index{localization, reflective}
\end{definition}

Thus, a localization of $\mathcal{C}$ at $W$ is the universal solution to inverting the elements of $W$. If the localization is reflective, then $\mathcal{D}$ can be thought of as a full subcategory of $\mathcal{C}$. Such reflective localizations are closely related to left Bousfield localizations in the context of model categories.

In the setting of interest to us here, the existence of rationalizations described above implies that the functor
\begin{equation*}
\mathcal{S}_*^{\geq 2} \rightarrow \mathcal{S}_{\mathbb{Q}}^{\geq 2}: X \mapsto X_{\mathbb{Q}}
\end{equation*}
is a reflective localization of $\mathcal{S}_*^{\geq 2}$ at the class of rational equivalences. An entirely analogous procedure produces the localization
\begin{equation*}
\mathrm{Sp} \rightarrow \mathrm{Sp}_{\mathbb{Q}}: E \mapsto E_{\mathbb{Q}}
\end{equation*}
of the $\infty$-category of spectra at the rational equivalences. One can explicitly identify the latter functor as $E_{\mathbb{Q}} \simeq H\mathbb{Q} \otimes E$. The $\infty$-category $\mathrm{Sp}_{\mathbb{Q}}$ is in fact equivalent to the $\infty$-category $\mathrm{Ch}_{\mathbb{Q}}$ of rational chain complexes and this equivalence respects the usual symmetric monoidal structures (cf. \cite{schwedeshipley} and Theorem 7.1.2.13 of \cite{ha}). Under these identifications, the rational spectrum $(\Sigma^\infty X)_{\mathbb{Q}}$ corresponds (up to natural equivalence) to the reduced rational chains $\widetilde{C}_*(X; \mathbb{Q})$. Applying the universal property of $\mathcal{S}_{\mathbb{Q}}^{\geq 2}$, the functor
\begin{equation*}
\mathcal{S}_*^{\geq 2} \rightarrow \mathrm{Sp}_{\mathbb{Q}}: X \mapsto (\Sigma^\infty X)_{\mathbb{Q}}
\end{equation*}
factors over a functor for which we write
\begin{equation*}
\Sigma^\infty_{\mathbb{Q}}\colon \mathcal{S}_{\mathbb{Q}}^{\geq 2} \rightarrow \mathrm{Sp}_{\mathbb{Q}}.
\end{equation*}
We will not drag the equivalence between $\mathrm{Sp}_{\mathbb{Q}}$ and $\mathrm{Ch}_{\mathbb{Q}}$ along and use $\Sigma^\infty_{\mathbb{Q}}$ as our version of the (reduced) rational chains functor.

The $\infty$-category $\mathcal{S}_{\mathbb{Q}}^{\geq 2}$ carries two evident symmetric monoidal structures, namely the Cartesian product and the rationalization of the smash product, although the second structure does not have a unit (since we insisted that our spaces be simply-connected). Any rational space $X$ is a commutative coalgebra with respect to the Cartesian product using the diagonal. The natural map from product to smash product then also makes $X$ a commutative coalgebra with respect to smash product (although without a counit). Since $\Sigma^\infty$ is a symmetric monoidal functor with respect to smash product, we conclude that $\Sigma^\infty_{\mathbb{Q}}$ may be factored as a composition
\begin{equation*}
\mathcal{S}_{\mathbb{Q}}^{\geq 2} \xrightarrow{\widetilde{C}_{\mathbb{Q}}} \mathrm{coCAlg}^{\mathrm{nu}}(\mathrm{Sp}_{\mathbb{Q}}) \rightarrow \mathrm{Sp}_{\mathbb{Q}},
\end{equation*}
where the second arrow is the forgetful functor. The middle term denotes the $\infty$-category of commutative coalgebras without counit in the symmetric monoidal $\infty$-category $\mathrm{Sp}_{\mathbb{Q}}$. For us a commutative coalgebra without counit in a symmetric monoidal $\infty$-category $\mathcal{C}$ is by definition a non-unital commutative algebra object of $\mathcal{C}^{\mathrm{op}}$ (cf. Definition 5.3 of \cite{heuts} or Definition 3.1.1 of \cite{lurieelliptic}).
\index{coalgebra, commutative}

\begin{remark}
There are often good 1-categorical models for $\infty$-categories of commutative algebra objects; for example, the $\infty$-category $\mathrm{CAlg}(\mathrm{Sp})$ of commutative ring spectra can be built from a model category of commutative monoids (in the usual 1-categorical sense) in any good symmetric monoidal model category of spectra. The situation for coalgebras is much less pleasant; we will ignore the issue and work with the above definition.
\end{remark}

We say a coalgebra in $\mathrm{Sp}_{\mathbb{Q}}$ is $n$-connected if its underlying spectrum is $n$-connected. One half of Theorem \ref{thm:rationalhomotopy} may then be rephrased as follows:

\begin{theorem}[Quillen \cite{quillen}]
\label{thm:quillen2}
The functor $\widetilde{C}_{\mathbb{Q}}$ gives an equivalence of $\infty$-categories $\mathcal{S}_{\mathbb{Q}}^{\geq 2} \rightarrow \mathrm{coCAlg}^{\mathrm{nu}}(\mathrm{Sp}_{\mathbb{Q}})^{\geq 2}$.
\end{theorem}

We will describe two proofs of this theorem; the first is close to the traditional proofs, the second uses Goodwillie calculus and does not appear to be in the literature.

\begin{proof}[Sketch of first proof of Theorem \ref{thm:quillen2}]
To check that $\widetilde{C}_{\mathbb{Q}}$ is fully faithful, first note that it admits a right adjoint $\mathrm{Map}_{\mathrm{coCAlg}}(H\mathbb{Q}, -)$, where $H\mathbb{Q}$ is equipped with the trivial coalgebra structure. We will argue that for any $X \in \mathcal{S}_{\mathbb{Q}}^{\geq 2}$ the unit map of this adjunction
\begin{equation*}
X \rightarrow \mathrm{Map}_{\mathrm{coCAlg}}(H\mathbb{Q}, \widetilde{C}_{\mathbb{Q}} X)
\end{equation*}
is an equivalence. Here and above, the mapping spaces refer to those in the $\infty$-category $\mathrm{coCAlg}^{\mathrm{nu}}(\mathrm{Sp}_{\mathbb{Q}})$. 

We work by induction on the Postnikov tower of $X$. Write $\tau_{\leq n} X$ for the $n$th Postnikov section. Since $X$ is simply-connected, its Postnikov tower consists of principal fibrations of the form
\begin{equation*}
\tau_{\leq n} X \rightarrow \tau_{\leq n-1} X \rightarrow K(\pi_n X, n+1).
\end{equation*}
The convergence of the homological Eilenberg--Moore spectral sequence for this fibration implies that $\widetilde{C}_{\mathbb{Q}}$ sends this fiber sequence to a fiber sequence of coalgebras. The reader might be more familiar with the dual statement that the cohomological version of this spectral sequence gives, for a suitable fibration $F \rightarrow E \rightarrow B$, an equivalence of commutative $H\mathbb{Q}$-algebras
\begin{equation*}
H\mathbb{Q} \otimes_{C^*(B; \mathbb{Q})} C^*(E; \mathbb{Q}) \rightarrow C^*(F; \mathbb{Q}),
\end{equation*}
but working with cochains would require us to introduce finite type hypotheses. In any case, we may now use induction to reduce to the case where $X$ is an Eilenberg--MacLane space $K(V,n)$ for $V$ a rational vector space. Then $\widetilde{C}_{\mathbb{Q}}(K(V,n))$ is the cofree commutative coalgebra (without counit) generated by $\Sigma^n HV$ (see Remark \ref{rmk:cofree}), the latter denoting the rational spectrum whose single nonvanishing homotopy group is $V$ in dimension $n$. But clearly
\begin{equation*}
\mathrm{Map}_{\mathrm{coCAlg}}(H\mathbb{Q}, \mathrm{cofree}(\Sigma^n HV)) \simeq \mathrm{Map}_{\mathrm{Sp}_{\mathbb{Q}}}(H\mathbb{Q}, \Sigma^n HV) \simeq K(V,n),
\end{equation*}
so that $\widetilde{C}_{\mathbb{Q}}$ is indeed fully faithful.

To see it is essentially surjective, one observes that $\mathrm{coCAlg}^{\mathrm{nu}}(\mathrm{Sp}_{\mathbb{Q}})^{\geq 2}$ is generated under colimits by a trivial coalgebra $C_2$ on a single generator in degree 2; indeed, any simply-connected coalgebra can be built from `cells' (where in this case cell means trivial coalgebra on a generator in degree $\geq 2$) in the same way that a space can be built from spheres. Since $\widetilde{C}_{\mathbb{Q}}$ preserves colimits and $C_2$ is the image of the rational 2-sphere $S^2_{\mathbb{Q}}$, this completes the argument.
\end{proof}

\begin{remark}
\label{rmk:cofree}
\index{coalgeba, cofree}
In general cofree coalgebras are rather hard to understand, in contrast to the case of free algebras (for which there is a simple formula). The difference can be traced back to the fact that the smash product of spectra or tensor product of chain complexes commutes with colimits in each variable separately, but not with limits. However, for degree reasons the cofree commutative coalgebra on a connected rational spectrum $E$ is described by the `naive' formula
\begin{equation*}
\mathrm{cofree}(E) \cong \bigvee_{k \geq 1} (E^{\otimes k})^{h\Sigma_k}.
\end{equation*}
Note that since we work over the rational numbers, it is not important whether we use homotopy fixed points or orbits on the right-hand side. Also note that if $E = \Sigma^n HV$ for a rational vector space $V$, this formula describes precisely the rational homology of $K(V,n)$. This fact was used in the argument above.
\end{remark}

\begin{proof}[Sketch of second proof of Theorem \ref{thm:quillen2}]
We now outline an alternative argument using Goodwillie calculus and some ideas from the theory of descent (or comonadicity). We refer to the chapter of Arone--Ching in this volume for background on Goodwillie calculus. A short reminder on monads and comonads can be found in Section \ref{sec:monads}. The following sets the tone for some of the techniques we will employ in Section \ref{sec:Liealgebras}.

The adjoint pair $(\Sigma^\infty_{\mathbb{Q}}, \Omega^\infty)$ between $\mathrm{S}^{\geq 2}_{\mathbb{Q}}$ and $\mathrm{Sp}_{\mathbb{Q}}$ in particular gives a comonad $\Sigma^\infty_{\mathbb{Q}}\Omega^\infty$ on $\mathrm{Sp}_{\mathbb{Q}}$, and $\Sigma^\infty_{\mathbb{Q}}$ gives a comparison functor
\begin{equation*}
\mathcal{S}_{\mathbb{Q}}^{\geq 2} \rightarrow \mathrm{coAlg}_{\Sigma^\infty_{\mathbb{Q}}\Omega^\infty}(\mathrm{Sp}_{\mathbb{Q}}).
\end{equation*}
The right-hand side denotes the $\infty$-category of \emph{coalgebras} (sometimes also called \emph{left comodules}) for the comonad $\Sigma^\infty_{\mathbb{Q}}\Omega^\infty$. In fact, this comparison functor is an equivalence; one says that the adjunction is \emph{comonadic}. \index{comonadic adjunction} This follows immediately (by the dual of Lemma \ref{lem:monadicitycriterion} below) from the fact that for any $X \in \mathcal{S}_{\mathbb{Q}}^{\geq 2}$, the limit of the cosimplicial object $(\Omega^\infty \Sigma^\infty_{\mathbb{Q}})^{\bullet +1} X$ (which may be thought of as the Bousfield--Kan $H\mathbb{Q}$-resolution) recovers $X$, meaning that the natural map
\[
\begin{tikzcd}
X \ar{r} & \mathrm{Tot}\bigl(\Omega^\infty\Sigma^\infty_\mathbb{Q} X \ar[shift left]{r} \ar[shift right]{r} & (\Omega^\infty\Sigma^\infty_\mathbb{Q})^2 X \ar{l} \ar{r} \ar[shift left = 2]{r} \ar[shift right = 2]{r} & \cdots \ar[shift right]{l} \ar[shift left]{l} \bigr)
\end{tikzcd}
\]
is an equivalence. Indeed, for an Eilenberg-MacLane space $X = K(V,n)$, the fact that $X = \Omega^\infty \Sigma^n HV$ implies that this cosimplicial object admits an extra codegeneracy (sometimes called a contracting codegeneracy). For general simply-connected $X$ the conclusion follows from induction on its Postnikov tower.

\index{coalgebra, cofree}
The forgetful-cofree adjoint pair between $\mathrm{coCAlg}^{\mathrm{nu}}(\mathrm{Sp}_{\mathbb{Q}})^{\geq 2}$ and $\mathrm{Sp}_{\mathbb{Q}}$ is comonadic as well, the relevant comonad in this case being the cofree coalgebra functor described in Remark \ref{rmk:cofree}. The universal property of the cofree coalgebra and the fact that the functor $\Sigma^\infty_{\mathbb{Q}}$ takes values in commutative coalgebras imply that the counit $\varepsilon\colon \Sigma^\infty_{\mathbb{Q}}\Omega^\infty \rightarrow \mathrm{id}_{\mathrm{Sp}_{\mathbb{Q}}}$ induces a natural transformation 
\begin{equation*}
\gamma \colon \Sigma^\infty_{\mathbb{Q}}\Omega^\infty \rightarrow \mathrm{cofree}.
\end{equation*}
To prove the theorem it suffices to show that $\gamma$ is an equivalence. This follows easily from the well-known calculation of the homogeneous layers (or derivatives) of $\Sigma^\infty_{\mathbb{Q}}\Omega^\infty$, which states
\begin{equation*}
D_n(\Sigma^\infty_{\mathbb{Q}}\Omega^\infty)(E) \simeq (E^{\otimes n})_{h\Sigma_n},
\end{equation*}
together with the fact that the Goodwillie tower of $\Sigma^\infty_{\mathbb{Q}}\Omega^\infty$ converges on connected spectra (see for example Section 3.2 of \cite{kuhntate} and Example 2.14 of \cite{aroneching}). Alternatively, the fact that $\gamma$ is an equivalence is easily deduced from Theorem 2.28 of \cite{kuhngoodwillie}. 

\end{proof}

\begin{remark}
\label{rmk:Tatevanishing}
In our argument involving homogeneous functors we tacitly used again that $(E^{\otimes n})^{h\Sigma_n} \simeq (E^{\otimes n})_{h\Sigma_n}$ in the $\infty$-category $\mathrm{Sp}_{\mathbb{Q}}$ of rational spectra. This equivalence between fixed points and orbits will turn out to be a crucial feature when establishing variations on rational homotopy theory later on.
\end{remark}

\begin{remark}
\label{rmk:SigmaOmega}
The argument above works just as well to prove the integral statement
\begin{equation*}
\mathcal{S}_*^{\geq 2} \simeq \mathrm{coAlg}_{\Sigma^\infty\Omega^\infty}(\mathrm{Sp})^{\geq 2}.
\end{equation*}
It is only in order to explicitly identify the comonad $\Sigma^\infty_{\mathbb{Q}}\Omega^\infty$ as the cofree commutative coalgebra that we used we are working over the rational numbers.
\end{remark}

\begin{remark}
\label{rmk:minimalmodels}
So far we have taken a rather abstract perspective. However, much of the power of rational homotopy theory derives from its computability. The linear (or Spanier-Whitehead) dual of the functor $C_{\mathbb{Q}}$ gives a functor from $\mathcal{S}_{\mathbb{Q}}^{\geq 2}$ to the $\infty$-category of commutative $H\mathbb{Q}$-algebras or, equivalently, the $\infty$-category of rational commutative differential graded algebras (cdga's). The latter functor can be described explicitly using Sullivan's functor $A_{\mathrm{PL}}$ of \emph{polynomial differential forms} \cite{sullivan}. Moreover, every cdga arising in this way is equivalent to a \emph{minimal} cdga $A$, which is one whose underlying graded commutative algebra is free and for which the differential takes values in decomposable elements. Two such minimal cdga's are equivalent if and only if they are isomorphic, making the theory very rigid. For a rational space $X$, a minimal replacement of $A_{\mathrm{PL}}(X)$ is called a \emph{minimal model} for $X$. These minimal models are a very powerful tool; for example, the rational homotopy groups of a space $X$ of finite type can be computed directly as the linear dual of the indecomposables of the corresponding minimal model, with the Whitehead bracket determined by the differential via a simple procedure.
\end{remark}

\index{differential graded Lie algebra}
The other half of Theorem \ref{thm:rationalhomotopy} concerns differential graded Lie algebras and arises as follows. If $\mathfrak{g}_*$ is a differential graded Lie algebra, then one can associate to it the (reduced) Chevalley--Eilenberg complex $\mathrm{CE}(\mathfrak{g}_*)$ computing its Lie algebra homology. The underlying graded vector space of $\mathrm{CE}(\mathfrak{g}_*)$ is
\begin{equation*}
\mathrm{Sym}^{\geq 1}(\mathfrak{g}_*[1]) := \bigoplus_{n \geq 1} \mathrm{Sym}^n(\mathfrak{g}_*[1]),
\end{equation*}
where $\mathfrak{g}_*[1]$ is the shift defined by $\mathfrak{g}_i[1] := \mathfrak{g}_{i-1}$. Here $\mathrm{Sym}^n(V)$ is the $n$th symmetric power $(V^{\otimes n})_{\Sigma_n}$ in the graded sense, i.e., it behaves like the exterior power on odd degree elements. The differential $d_{\mathrm{CE}}$ is constructed from the differential $d$ of $\mathfrak{g}$ and the Lie bracket (see Appendix B.6 of \cite{quillen}). In fact, $\mathrm{Sym}^{\geq 1}(\mathfrak{g}_*[1])$ carries an evident coalgebra structure that is compatible with the differential just defined, in which the primitive elements are precisely $\mathfrak{g}_*[1] \subseteq \mathrm{Sym}^{\geq 1}(\mathfrak{g}_*[1])$. (An element $x$ of a coalgebra without counit is primitive if $\Delta x = 0$.) This makes $\mathrm{CE}(\mathfrak{g}_*)$ a differential graded commutative coalgebra without counit. It is not hard to verify that the functor $\mathrm{CE}$ preserves quasi-isomorphisms of differential graded Lie algebras. The following then establishes the remaining half of Theorem \ref{thm:rationalhomotopy}:


\begin{theorem}[Quillen \cite{quillen}]
\label{thm:CE}
The functor
\begin{equation*}
\mathrm{CE} \colon \mathrm{Lie}(\mathrm{Ch}_{\mathbb{Q}})^{\geq 1} \rightarrow \mathrm{coCAlg}^{\mathrm{nu}}(\mathrm{Ch}_{\mathbb{Q}})^{\geq 2}
\end{equation*}
is an equivalence of $\infty$-categories.
\end{theorem}

We will place this theorem in a more general context in Section \ref{sec:koszul}. The inverse functor takes the derived primitives of a commutative coalgebra (shifted down in degree by 1). We conclude this section by highlighting one observation about Theorem \ref{thm:rationalhomotopy}. Let us write
\begin{equation*}
\Phi_0 \colon \mathcal{S}_{\mathbb{Q}}^{\geq 2} \rightarrow \mathrm{Sp}_{\mathbb{Q}}^{\geq 1}
\end{equation*}
for the composition of the functor $L_{\mathbb{Q}}$ with the forgetful functor from differential graded Lie algebras to $\mathrm{Sp}_{\mathbb{Q}}$ (which we have identified with $\mathrm{Ch}_{\mathbb{Q}}$). Also, write $\Theta_0$ for its left adjoint, which is the composition of the free Lie algebra functor with the inverse of the equivalence $L_{\mathbb{Q}}$. The reason for this seemingly strange notation is an analogy with the Bousfield--Kuhn functor\index{Bousfield--Kuhn functor}, which we discuss in Section \ref{sec:periodicity}. We now have two different adjunctions between the $\infty$-categories of rational spaces and of rational spectra, summarized in the following diagram (left adjoints on top):
\[
\begin{tikzcd}
\mathrm{Sp}_{\mathbb{Q}}^{\geq 1} \ar[shift left]{r}{\Theta_0} & \mathcal{S}_{\mathbb{Q}}^{\geq 2} \ar[shift left]{l}{\Phi_0} \ar[shift left]{r}{\Sigma^\infty_{\mathbb{Q}}} & \mathrm{Sp}_{\mathbb{Q}}^{\geq 2} \ar[shift left]{l}{\Omega^\infty}.
\end{tikzcd}
\]
The horizontal composition from right to left computes the derived primitives of the cofree coalgebra, shifted down by one. In fact it is rather easy to see that for a rational vector space $V$, the primitives of $\mathrm{cofree}(V)$ are precisely $V$. From this one can deduce that $\Phi_0 \Omega^\infty$ is precisely the functor $\Sigma^{-1}$ that shifts down by one. Consequently, the composition of left adjoints $\Sigma^\infty_{\mathbb{Q}}\Theta_0$ is $\Sigma$, the functor shifting up by one. Alternatively, one can verify directly that the Chevalley--Eilenberg homology of a free Lie algebra on a chain complex $V$ is quasi-isomorphic to $V[1]$. Variations on these observations will play an important role in Sections \ref{sec:periodicity} and \ref{sec:Liealgebras}.

\section{Monads and their algebras}
\label{sec:monads}

\index{monad}
As shown by the second proof of Theorem \ref{thm:quillen2}, the yoga of monads and comonads can be very useful when attempting to `model' a given $\infty$-category $\mathcal{C}$ by some homotopy theory of algebras or coalgebras. We already used it to identify the $\infty$-category of simply-connected pointed spaces with that of coalgebras for the comonad $\Sigma^\infty\Omega^\infty$. Later on, we will relate localizations of the $\infty$-category of spaces to algebras over the spectral Lie operad using similar techniques. To help the reader we offer this short section, which collects some basic facts about (co)monads and the associated (co)simplicial objects. We will make systematic use of these throughout the rest of this survey. More background can be found in Section 4.7 of \cite{ha} and in \cite{riehlverity}.

Let $\mathcal{C}$ be an $\infty$-category. Then the $\infty$-category $\mathrm{Fun}(\mathcal{C},\mathcal{C})$ admits a monoidal structure, with tensor product given by composition of functors. A \emph{monad} (resp. \emph{comonad}) is a monoid (resp. comonoid) in this monoidal $\infty$-category. If
\[
\begin{tikzcd}
\mathcal{C} \ar[shift left]{r}{F} & \mathcal{D} \ar[shift left]{l}{G}
\end{tikzcd}
\]
is an adjoint pair, with $F$ left adjoint, then the composite $GF$ gives a monad on $\mathcal{C}$ (and dually the functor $FG$ gives a comonad on $\mathcal{D}$). Writing $\eta: \mathrm{id}_{\mathcal{C}} \rightarrow GF$ for the unit and $\varepsilon: FG \rightarrow \mathrm{id}_{\mathcal{D}}$ for the counit of the adjunction, the multiplication and unit maps of $GF$ are given by
\begin{equation*}
GFGF \xrightarrow{G\varepsilon F} GF \quad\quad \text{and} \quad\quad \mathrm{id}_{\mathcal{C}} \xrightarrow{\eta} GF.
\end{equation*}
In ordinary category theory these maps would then have to satisfy an associativity and a unitality condition. In the setting of $\infty$-categories, these constraints become extra data rather than properties; a monoid object can be encoded by a certain diagram of shape $\mathbf{\Delta}^{\mathrm{op}}$, cf. Section 4.1 of \cite{ha}.

A typical example to have in mind is the one where $\mathcal{D}$ is the $\infty$-category of commutative algebras in $\mathcal{C}$ (and we assume $\mathcal{C}$ to be a sufficiently nice symmetric monoidal $\infty$-category), with $G$ the forgetful functor and $F$ the free commutative algebra functor. The monad $GF$ then assigns to an object $X$ of $\mathcal{C}$ the underlying object of the free commutative algebra on $X$. An important example of a comonad, already used in the previous section, is the functor $\Sigma^\infty\Omega^\infty$ on the $\infty$-category of spectra.

\index{algebra for a monad}
If $T$ is a monad on an $\infty$-category $\mathcal{C}$, one can speak of \emph{$T$-algebras} in $\mathcal{C}$. Such an algebra is an object $X$ which is first of all equipped with an `action'
\begin{equation*}
\mu\colon TX \rightarrow X.
\end{equation*}
Again, in ordinary category this map would then have to satisfy an associativity condition (with respect to the multiplication of $T$) and a unitality condition (using the unit map of $T$). Working with $\infty$-categories means one has to encode these constraints in a coherent way; we refer to Section 4.2 of \cite{ha} for details. One important point for us is that any $T$-algebra $X$ in particular gives rise to an augmented simplicial object in the $\infty$-category of $T$-algebras of the form
\[
\begin{tikzcd}
\cdots \ar{r} \ar[shift left = 2]{r} \ar[shift right = 2]{r} & T^2 X \ar[shift right]{l} \ar[shift left]{l} \ar[shift left]{r} \ar[shift right]{r} & TX \ar{r} \ar{l} & X.
\end{tikzcd}
\]
The degeneracy maps use the unit of $T$, while the face maps describe the monoid multiplication of $T$ and the action map $TX \rightarrow X$. This augmented simplicial object is canonically contractible when considered as a diagram in the underlying $\infty$-category $\mathcal{C}$. Indeed it admits `extra degeneracies' $X \rightarrow TX$, and similarly $T^{\bullet} X \rightarrow T^{\bullet +1} X$, given by the unit of $T$. It follows that the diagram above expresses $X$ as the colimit of the simplicial object $T^{\bullet +1} X$, thus describing $X$ as a colimit of a diagram of \emph{free} $T$-algebras. We write $\mathrm{Alg}_T(\mathcal{C})$ for the $\infty$-category of $T$-algebras in $\mathcal{C}$. Dually, if $Q$ is a comonad on $\mathcal{C}$, we write $\mathrm{coAlg}_Q(\mathcal{C})$ for the $\infty$-category of $Q$-coalgebras.

\begin{remark}
\label{rmk:leftTmodule}
The definition of a $T$-algebra is really a special instance of the definition of a left module $M$ in an $\infty$-category $\mathcal{C}$ for an associative algebra $A$ in some monoidal $\infty$-category $\mathcal{D}$ that acts on $\mathcal{C}$. In our case $T$ plays the role of $A$ and $\mathcal{D} = \mathrm{Fun}(\mathcal{C},\mathcal{C})$. In fact, this interpretation as modules also inspires some notation which will be useful. For an associative algebra $A$ with a right module $M$ and left module $N$ one can form the \emph{two-sided bar construction}, which is the simplicial object
\[
\begin{tikzcd}
\cdots M \otimes A^{\otimes 2} \otimes N \ar{r} \ar[shift left = 2]{r} \ar[shift right = 2]{r} & M \otimes A \otimes N \ar[shift right]{l} \ar[shift left]{l} \ar[shift left]{r} \ar[shift right]{r} & M \otimes N \ar{l}
\end{tikzcd}
\]
for which we write $\mathrm{Bar}(M, A, N)_{\bullet} := M \otimes A^{\otimes \bullet} \otimes N$. The face maps use the action of $A$ on $M$ and $N$ for the outer faces and the multiplication of $A$ for the inner faces. The colimit of $\mathrm{Bar}(M, A, N)_{\bullet}$ computes the \emph{relative tensor product} $M \otimes_A N$. With this notation, the simplicial object we used above the remark can be written $\mathrm{Bar}(T,T,X)_\bullet$. \index{bar construction}
\end{remark}

For any adjoint pair
\[
\begin{tikzcd}
\mathcal{C} \ar[shift left]{r}{F} & \mathcal{D} \ar[shift left]{l}{G},
\end{tikzcd}
\]
giving a monad $T := GF$ on $\mathcal{C}$, the functor $G$ can canonically be factored as
\begin{equation*}
\mathcal{D} \xrightarrow{\varphi} \mathrm{Alg}_T(\mathcal{C}) \xrightarrow{\mathrm{forget}} \mathcal{C}.
\end{equation*}
Indeed, for $X \in \mathcal{D}$ the object $G(X)$ has a natural $T$-algebra structure with action map given by
\begin{equation*}
TG(X) = GFG(X) \xrightarrow{G\varepsilon} G(X).
\end{equation*}
One says that the adjoint pair $(F,G)$ is \emph{monadic} if the comparison functor $\varphi: \mathcal{D} \rightarrow \mathrm{Alg}_T(\mathcal{C})$ is an equivalence of $\infty$-categories. The Barr--Beck theorem is a surprisingly useful recognition criterion for monadic adjunctions, generalized to higher category theory by Lurie \cite{ha} (see \cite{riehlverity} for an alternative approach):
\index{monadic adjunction}

\index{monadicity}
\begin{theorem}[Barr--Beck, Lurie]
\label{thm:barrbeck}
An adjoint pair $F: \mathcal{C} \rightleftarrows \mathcal{D}: G$ is monadic if and only if the following two conditions are satisfied:
\begin{itemize}
\item[(a)] The functor $G$ is \emph{conservative}, i.e., a morphism $f$ in $\mathcal{D}$ is an equivalence if and only if $G(f)$ is an equivalence.
\item[(b)] If $X_{\bullet}$ is a simplicial object in $\mathcal{D}$ such that $G(X_{\bullet})$ is split (i.e., admits an augmentation with extra degeneracies), then the colimit of $X_{\bullet}$ exists in $\mathcal{D}$ and is preserved by $G$. 
\end{itemize}
\end{theorem}

Condition (b) might seem cryptic at first sight: however, it is automatically satisfied if $\mathcal{D}$ admits all geometric realizations (i.e., colimits of $\mathbf{\Delta}^{\mathrm{op}}$-shaped diagrams) and $G$ preserves these, which will suffice for our applications. We usually denote the colimit of a simplicial object $X_{\bullet}$ by $|X_{\bullet}|$. The typical split simplicial object to have in mind is the diagram above Remark \ref{rmk:leftTmodule} describing $T^{\bullet + 1} X$.

Any adjoint pair $F: \mathcal{C} \rightleftarrows \mathcal{D}: G$ gives rise to a simplicial resolution of the identity functor of $\mathcal{D}$. To be precise, it gives a simplicial object $(FG)^{\bullet + 1}$ with an augmentation to $\mathrm{id}_{\mathcal{D}}$:
\[
\begin{tikzcd}
\cdots \ar{r} \ar[shift left = 2]{r} \ar[shift right = 2]{r} & (FG)^2 \ar[shift right]{l} \ar[shift left]{l} \ar[shift left]{r} \ar[shift right]{r} & FG \ar{r}{\varepsilon} \ar{l} & \mathrm{id}_{\mathcal{D}}.
\end{tikzcd}
\]
In the notation of bar constructions this simplicial object can be written as $\mathrm{Bar}(F, GF, G)_{\bullet}$. The degeneracy maps use the unit $\eta\colon \mathrm{id}_{\mathcal{C}} \rightarrow GF$, the face maps use the counit $\varepsilon$. A dual comment applies to give a cosimplicial resolution of the identity functor of $\mathcal{C}$. We used exactly this resolution, associated to the adjunction $(\Sigma^\infty, \Omega^\infty)$, in the second proof of Theorem \ref{thm:quillen2}. The following is a useful criterion for monadicity:

\begin{lemma}
\label{lem:monadicitycriterion}
Suppose $F: \mathcal{C} \rightleftarrows \mathcal{D}: G$ is an adjoint pair and that $\mathcal{D}$ admits colimits of $G$-split simplicial objects. Then this pair is monadic if and only if for every object $X$ of $\mathcal{C}$, the map
\begin{equation*}
|(FG)^{\bullet + 1}|(X) \rightarrow X
\end{equation*}
arising from the simplicial resolution described above is an equivalence. 
\end{lemma}

\section{Koszul duality}
\label{sec:koszul}

\index{Koszul duality}
The duality between Lie algebras and commutative coalgebras expressed by Theorem \ref{thm:CE} is a special case of what is usually called \emph{Koszul duality} or \emph{bar-cobar duality}. We will begin this section with a discussion of this duality for associative algebras and coalgebras. In the setting of differential graded algebras this goes back at least to the work of Adams \cite{adamscobar}, Priddy \cite{priddy}, and Moore \cite{moore}; we will phrase it in a more general context following Lurie (Section 4.3 of \cite{luriedagX}). The remainder of the section is a discussion of Koszul duality in the generality of (co)algebras for (co)operads. References for this material are \cite{getzlerjones,ginzburgkapranov,lodayvallette} in the differential graded context, \cite{chingbarcobar} in the context of stable homotopy theory, and \cite{francisgaitsgory} for a discussion at the more general level of $\infty$-categories. 

For the purposes of this discussion, let $\mathcal{C}$ be a symmetric monoidal $\infty$-category. We will assume it admits limits and colimits. Suppose $A$ is an associative algebra object of $\mathcal{C}$ equipped with an augmentation $\varepsilon: A \rightarrow \mathbf{1}$ to the unit of $\mathcal{C}$. Then the \emph{bar construction} of $A$\index{bar construction}, denoted $\mathrm{Bar}(A)$, is the relative tensor product $\mathbf{1} \otimes_A \mathbf{1}$, which by definition is the colimit of the simplicial object
\[
\begin{tikzcd}
\cdots A \otimes A \ar{r} \ar[shift left = 2]{r} \ar[shift right = 2]{r} & A \ar[shift right]{l} \ar[shift left]{l} \ar[shift left]{r} \ar[shift right]{r} & \mathbf{1}. \ar{l}
\end{tikzcd}
\]
The degeneracy maps of this simplicial object are constructed from the unit of $A$, whereas the face maps use the multiplication $A \otimes A \rightarrow A$ for the `inner' faces and the augmentation $\varepsilon$ for the `outer' faces. In the notation of Remark \ref{rmk:leftTmodule} it can be written as $\mathrm{Bar}(\mathbf{1}, A, \mathbf{1})_\bullet$.

The object $\mathrm{Bar}(A)$ admits a comultiplication, using the maps
\begin{equation*}
\mathbf{1} \otimes_A \mathbf{1} \simeq \mathbf{1} \otimes_A A \otimes_A \mathbf{1} \xrightarrow{\varepsilon} \mathbf{1} \otimes_A \mathbf{1} \otimes_A \mathbf{1} \rightarrow (\mathbf{1} \otimes_A \mathbf{1}) \otimes (\mathbf{1} \otimes_A \mathbf{1}).
\end{equation*}
(The last map arises from the universal property of the colimit defining the term $\mathbf{1} \otimes_A \mathbf{1} \otimes_A \mathbf{1}$; under the additional assumption that the tensor product on $\mathcal{C}$ commutes with geometric realizations in each variable, it is always an equivalence.) However, the situation is much better than just this: the object $\mathrm{Bar}(A)$ can be upgraded to a homotopy-coherent associative coalgebra object of $\mathcal{C}$ equipped with a coaugmentation. Write $\mathrm{Alg}^{\mathrm{aug}}(\mathcal{C})$ for the $\infty$-category of augmented associative algebras in $\mathcal{C}$ and 
\begin{equation*}
\mathrm{coAlg}^{\mathrm{aug}}(\mathcal{C}) := \mathrm{Alg}^{\mathrm{aug}}(\mathcal{C}^{\mathrm{op}})^{\mathrm{op}}
\end{equation*}
for the $\infty$-category of coaugmented associative coalgebras. The generality of the following statement is first articulated by Lurie in \cite{luriedagX}, but its essential form traces back to Moore \cite{moore}:

\begin{theorem}
\label{thm:barcobar}
There is an adjoint pair of functors (in which $\mathrm{Bar}$ is left adjoint) as follows:
\[
\begin{tikzcd}
\mathrm{Alg}^{\mathrm{aug}}(\mathcal{C}) \ar[shift left]{r}{\mathrm{Bar}} & \mathrm{coAlg}^{\mathrm{aug}}(\mathcal{C}) \ar[shift left]{l}{\mathrm{Cobar}}.
\end{tikzcd}
\]
\end{theorem}

\index{cobar construction}
Here $\mathrm{Cobar}$ is the bar construction applied to the opposite category, i.e., it sends a coaugmented coalgebra $C$ to the totalization of the cosimplicial object 
\[
\begin{tikzcd}
\mathbf{1} \ar[shift left]{r} \ar[shift right]{r} & C  \ar{r} \ar[shift left = 2]{r} \ar[shift right = 2]{r} \ar{l} & C \otimes C \ar[shift right]{l} \ar[shift left]{l} \cdots.
\end{tikzcd}
\]
The proof of Theorem \ref{thm:barcobar} proceeds by showing that the two mapping spaces $\mathrm{Map}_{\mathrm{Alg}^{\mathrm{aug}}(\mathcal{C})}(A, \mathrm{Cobar}(C))$ and $\mathrm{Map}_{\mathrm{coAlg}^{\mathrm{aug}}(\mathcal{C})}(\mathrm{Bar}(A), C)$ are both equivalent to the space of lifts of the pair $(C,A)$, thought of as an augmented algebra object of the $\infty$-category $\mathcal{C}^{\mathrm{op}} \times \mathcal{C}$, to an augmented algebra object of the twisted arrow category $\mathrm{TwArr}(\mathcal{C})$. This is closely related to the more classical notion of twisting cochains.

\index{operad} \index{cooperad}
The aim of the rest of this section is to apply this general setup for bar-cobar duality to the case of operads and cooperads. We write $\mathrm{SymSeq}(\mathcal{C})$ for the $\infty$-category of symmetric sequences in $\mathcal{C}$; its objects are sequences of objects $\mathcal{O} = \{\mathcal{O}(n)\}_{n \geq 1}$ of $\mathcal{C}$ where the $n$th term $\mathcal{O}(n)$ is equipped with an action of the symmetric group $\Sigma_n$. To any such sequence we associate a functor
\begin{equation*}
F_{\mathcal{O}}\colon \mathcal{C} \rightarrow \mathcal{C}\colon X \mapsto \bigoplus_{n \geq 1} (\mathcal{O}(n) \otimes X^{\otimes n})_{h\Sigma_n}.
\end{equation*}
The $\infty$-category $\mathrm{SymSeq}(\mathcal{C})$ carries a monoidal structure, called the \emph{composition product} and usually denoted by $\circ$, which is essentially characterized by the fact that the assignment \index{composition product}
\begin{equation*}
\mathrm{SymSeq}(\mathcal{C}) \rightarrow \mathrm{Fun}(\mathcal{C}, \mathcal{C})\colon \mathcal{O} \mapsto F_{\mathcal{O}}
\end{equation*}
is monoidal (see Section 4.1.2 of \cite{brantner} or \cite{haugseng}). Here the $\infty$-category on the right is monoidal via the composition of functors. The unit for the composition product is the symmetric sequence $\mathbf{1}$, which has the monoidal unit of $\mathcal{C}$ as its first term and all other terms equal to $0$. In these terms, an operad (resp. a cooperad) in $\mathcal{C}$ is a monoid (resp. a comonoid) in $\mathrm{SymSeq}(\mathcal{C})$ with respect to the composition product.


Since we excluded the case $n=0$ from our definitions above, all operads and cooperads we consider here are in fact \emph{non-unital}. If $\mathcal{O}$ is an operad, then the associated functor $F_{\mathcal{O}}$ has the structure of a monad (and dually for a cooperad). One can then define an $\mathcal{O}$-algebra in $\mathcal{C}$ to be precisely an $F_{\mathcal{O}}$-algebra. 

\begin{remark}
\label{rmk:dpcoalgebras}
Consider a cooperad $\mathcal{O}$. Observe that coalgebras for the comonad $F_{\mathcal{O}}$ are objects $X$ of $\mathcal{C}$ equipped with a `comultiplication map'
\begin{equation*}
X \rightarrow \bigoplus_{n \geq 1} (\mathcal{O}(n) \otimes X^{\otimes n})_{h\Sigma_n}
\end{equation*}
and further data recording its coassociativity and counitality. In the special case where $\mathcal{O}$ is the commutative cooperad (which as a symmetric sequence has $\mathcal{O}(n)$ equal to the monoidal unit of $\mathcal{C}$ for each $n \geq 1$), this notion of coalgebra is generally \emph{not} the same as that of a commutative coalgebra. One important difference is that the $n$-fold comultiplication of a commutative coalgebra $C$ gives a map
\begin{equation*}
C \rightarrow (C^{\otimes n})^{h\Sigma_n},
\end{equation*}
but this map need not factor through $(C^{\otimes n})_{h\Sigma_n}$. For this reason, the coalgebras for the comonad $F_{\mathcal{O}}$ are sometimes called \emph{conilpotent divided power $\mathcal{O}$-coalgebras} (cf. \cite{francisgaitsgory}). The adjective conilpotent refers to the direct sum, the term divided powers refers to the fact that $F_{\mathcal{O}}$ involves coinvariants for the symmetric groups, rather than invariants. 
\end{remark}

If $f\colon \mathcal{O} \rightarrow \mathcal{P}$ is a map of operads in $\mathcal{C}$, one obtains an adjunction between the corresponding $\infty$-categories of algebras:
\[
\begin{tikzcd}
\mathrm{Alg}_{\mathcal{O}} \ar[shift left]{r}{f_!} & \mathrm{Alg}_{\mathcal{P}} \ar[shift left]{l}{f^*}.
\end{tikzcd}
\]
Here the right adjoint $f^*$ is restriction along $f$. Informally, the left adjoint is the relative tensor product $\mathcal{P} \circ_{\mathcal{O}} -$, using the analogy between operads with their algebras and rings with their modules. More precisely, one first observes that $f_!$ sends free $\mathcal{O}$-algebras to free $\mathcal{P}$-algebras. As we described in Section \ref{sec:monads}, any $\mathcal{O}$-algebra $X$ admits a natural simplicial resolution in terms of free $\mathcal{O}$-algebras, namely the bar construction $\mathrm{Bar}(F_{\mathcal{O}}, F_{\mathcal{O}}, X)_\bullet$. Then $f_!X$ can be computed as the colimit of the simplicial object
\[
\begin{tikzcd}
\bigl(\cdots F_{\mathcal{P}} F_{\mathcal{O}}^2 X \ar{r} \ar[shift left = 2]{r} \ar[shift right = 2]{r}& F_{\mathcal{P}} F_{\mathcal{O}} X \ar[shift right]{l} \ar[shift left]{l} \ar[shift left]{r} \ar[shift right]{r} & F_{\mathcal{P}} X \bigr) \ar{l} =: \mathrm{Bar}(F_{\mathcal{P}}, F_{\mathcal{O}}, X).
\end{tikzcd}
\]

Let us specialize to the case where the $\infty$-category $\mathcal{C}$ is stable and $\mathcal{O}$ is \emph{reduced}, meaning $\mathcal{O}(1)$ is the monoidal unit of $\mathcal{C}$. The latter assumption gives essentially unique morphisms of operads $i\colon \mathbf{1} \rightarrow \mathcal{O}$ and $p\colon \mathcal{O} \rightarrow \mathbf{1}$, including the first term and projecting onto it respectively. Note that we may identify $\mathrm{Alg}_{\mathbf{1}}(\mathcal{C})$ with $\mathcal{C}$ itself, since $F_{\mathbf{1}}$ is the identity functor of $\mathcal{C}$. Then $i^*$ is the forgetful functor and $i_!$ the free $\mathcal{O}$-algebra functor. The functor $p^*$ can be identified with the \emph{trivial algebra functor} $\mathrm{triv}_{\mathcal{O}}$, equipping an object $X$ of $\mathcal{C}$ with the $\mathcal{O}$-algebra structure for which all maps
\begin{equation*}
(\mathcal{O}(n) \otimes X^{\otimes n})_{h\Sigma_n} \rightarrow X
\end{equation*}
are zero when $n > 1$ (recall that $\mathcal{C}$ is stable, so has a zero object). The left adjoint $p_!$ computes the \emph{derived indecomposables} of an $\mathcal{O}$-algebra $X$. We will denote it by $\mathrm{TAQ}_{\mathcal{O}}$, which stands for \emph{topological Andr\'{e}--Quillen homology}. \index{topological Andr\'{e}--Quillen homology} The reason for this terminology is that in the special case $\mathcal{O} = \mathbf{Com}$, this functor is closely related to the homology of commutative rings as defined in terms of the cotangent complex by Andr\'{e} and Quillen \cite{quillencotangent}. The two adjunctions just discussed give a diagram
\[
\begin{tikzcd}
\mathcal{C} \ar[shift left]{r}{\mathrm{free}_{\mathcal{O}}} & \mathrm{Alg}_{\mathcal{O}}(\mathcal{C}) \ar[shift left]{l}{\mathrm{forget}} \ar[shift left]{r}{\mathrm{TAQ}_{\mathcal{O}}} & \mathcal{C} \ar[shift left]{l}{\mathrm{triv}_{\mathcal{O}}}
\end{tikzcd}
\]
in which both horizontal composites are equivalent to the identity (because $pi = \mathrm{id}$).

In fact the functor $\mathrm{TAQ}_{\mathcal{O}}$ can be characterized by a universal property:

\begin{theorem}[Basterra--Mandell \cite{basterramandell}, Lurie {\cite[Theorem~7.3.4.7]{ha}}]
\label{thm:basterramandell}
Suppose the symmetric monoidal $\infty$-category $\mathcal{C}$ is stable, presentable, and the tensor product preserves colimits in each variable separately. Then the adjoint pair $(\mathrm{TAQ}_{\mathcal{O}}, \mathrm{triv}_{\mathcal{O}})$ exhibits $\mathcal{C}$ as the stabilization of $\mathrm{Alg}_{\mathcal{O}}(\mathcal{C})$. In other words, $\mathrm{TAQ}_{\mathcal{O}}$ is the initial colimit-preserving functor from $\mathrm{Alg}_{\mathcal{O}}(\mathcal{C})$ to a presentable stable $\infty$-category.
\end{theorem}
\begin{proof}[Sketch of proof]
Take the diagram above the statement of the theorem, stabilize the $\infty$-category $\mathrm{Alg}_{\mathcal{O}}(\mathcal{C})$, and replace the functors by their linearizations (or first derivatives) to obtain the following diagram:
\[
\begin{tikzcd}
\mathcal{C} \ar[shift left]{r}{\partial\mathrm{free}_{\mathcal{O}}} & \mathrm{Sp}(\mathrm{Alg}_{\mathcal{O}}(\mathcal{C})) \ar[shift left]{l}{\partial\mathrm{forget}} \ar[shift left]{r}{\partial\mathrm{TAQ}_{\mathcal{O}}} & \mathcal{C}. \ar[shift left]{l}{\partial\mathrm{triv}_{\mathcal{O}}}
\end{tikzcd}
\]
The chain rule $\partial(G \circ F) \cong \partial(G) \circ \partial (F)$ guarantees that the horizontal composites are still equivalent to the identity. Moreover, it is a general fact (and not hard to verify) that the stabilization of a monadic adjunction is monadic, so that the pair of functors on the left exhibits $\mathrm{Sp}(\mathrm{Alg}_{\mathcal{O}}(\mathcal{C}))$ as monadic over $\mathcal{C}$. The relevant monad is the linearization of the functor $\mathrm{forget} \circ \mathrm{free}_{\mathcal{O}}$; the latter is explicitly given by assigning to $X \in \mathcal{C}$ the object
\begin{equation*}
\bigoplus_{n \geq 1} (\mathcal{O}(n) \otimes X^{\otimes n})_{h\Sigma_n}.
\end{equation*}
The linearization of this clearly is the monad
\begin{equation*}
X \mapsto \mathcal{O}(1) \otimes X \cong X,
\end{equation*}
from which it follows that $\partial \mathrm{free}_{\mathcal{O}}\colon \mathcal{C} \rightarrow \mathrm{Sp}(\mathrm{Alg}_{\mathcal{O}}(\mathcal{C}))$ is an equivalence. This implies the theorem.
\end{proof}

\begin{remark}
One can paraphrase Theorem \ref{thm:basterramandell} by saying that $\mathrm{TAQ}_{\mathcal{O}}$ is the universal homology theory for $\mathcal{O}$-algebras. For specific choices of $\mathcal{O}$ it reproduces familiar notions of homology. As an example, take $\mathcal{C}$ to be the $\infty$-category $\mathrm{Ch}_{\mathbb{Z}}$ of chain complexes of abelian groups. As alluded to above, $\mathrm{TAQ}_{\mathbf{Com}}$ reproduces Andr\'{e}--Quillen homology of commutative rings. For associative algebras it gives a degree shift of Hochschild homology (when working over a field) or of Shukla homology (for general rings). For Lie algebras it produces a degree shift of the Chevalley--Eilenberg homology already discussed in Section \ref{sec:rationalhomotopy}. The fact that $\mathrm{TAQ}_{\mathcal{O}}$ preserves colimits, together with the equivalence $\mathrm{TAQ}_{\mathcal{O}} \circ \mathrm{free}_{\mathcal{O}} \cong \mathrm{id}_{\mathcal{C}}$, makes this homology theory a useful tool to understand the `cell structure' of an algebra, meaning the way that $X$ can be built from free algebras on generators of $\mathcal{C}$ (the cells) by pushouts (the cell attachments). This is similar to how singular homology can be used to investigate the minimal cell structure of a topological space. In case $\mathcal{O}$ is the operad $\mathbf{E}_n$ of little $n$-discs, this perspective is used to great effect in \cite{galatiuskupersrw} and its sequels.
\end{remark}

\index{Koszul duality}
The starting point of Koszul duality from the point of view of operads, which originates in \cite{ginzburgkapranov} and \cite{getzlerjones}, is the following. Again consider an operad $\mathcal{O}$ for which $\mathcal{O}(1)$ is the monoidal unit of $\mathcal{C}$, so that projection to this term gives an augmentation $\varepsilon\colon \mathcal{O} \rightarrow \mathbf{1}$. Then $\mathcal{O}$ becomes an augmented associative algebra in $\mathrm{SymSeq}(\mathcal{C})$ and thus it makes sense to speak of its bar construction $\mathrm{Bar}(\mathcal{O})$, which is a cooperad. 

\begin{example}
\label{ex:operads}
One of the most important examples is the duality between the commutative operad and the Lie operad. To be precise, the results of Ginzburg--Kapranov \cite{ginzburgkapranov} and Getzler--Jones \cite{getzlerjones} show that the cobar construction of the commutative cooperad in chain complexes of $R$-modules (for some commutative ring $R$) give the Lie operad with a degree shift. Dually, the bar construction of the Lie operad produces the commutative cooperad, again with a shift. In Section \ref{sec:Lie} we will return to a variant of this example in the context of stable homotopy theory in order to define spectral Lie algebras. Another example is the bar construction of the associative operad, which is the linear dual of the associative operad with a degree shift. More generally, Fresse \cite{fresse} proves that the bar construction of the operad given by the singular chains of the little discs operad $\mathbf{E}_n$ is the cooperad $C^*\mathbf{E}_n^{\vee}$ (shifted $n$ times). This is what is often referred to as the \emph{self-duality} of the $\mathbf{E}_n$-operad. Conjecturally this self-duality already holds for the operad $\Sigma^\infty_+\mathbf{E}_n$ in the $\infty$-category of spectra. At the time of writing this has not yet been resolved, but at least the underlying symmetric sequence of the cooperad $\mathrm{Bar}(\Sigma^\infty_+\mathbf{E}_n)$ has the expected homotopy type \cite{chingbarcobar}.
\end{example}

\begin{remark}
In fact, for operads and cooperads in a stable $\infty$-category $\mathcal{C}$ for which the tensor product is exact in each variable separately, the process of forming bar and cobar constructions is invertible; the unit and counit maps
\begin{equation*}
\mathcal{O} \rightarrow \mathrm{Cobar}(\mathrm{Bar}(\mathcal{O})) \quad\quad \mathrm{Bar}(\mathrm{Cobar}(\mathcal{Q})) \rightarrow \mathcal{Q}
\end{equation*}
are equivalences (cf. \cite{francisgaitsgory} and \cite{chingbarcobar}).
\end{remark}

We write $\mathrm{coAlg}^{\mathrm{dp}}_{\mathrm{Bar}(\mathcal{O})}(\mathcal{C})$ for the $\infty$-category of conilpotent divided power $\mathrm{Bar}(\mathcal{O})$-coalgebras (cf. Remark \ref{rmk:dpcoalgebras}). More briefly, this is the $\infty$-category of coalgebras for the associated comonad $F_{\mathrm{Bar}(\mathcal{O})}$.

\begin{theorem}[Francis--Gaitsgory \cite{francisgaitsgory}]
\label{thm:koszul1}
The functor $\mathrm{TAQ}_{\mathcal{O}}$ factors as a composite
\begin{equation*}
\mathrm{Alg}_{\mathcal{O}}(\mathcal{C}) \xrightarrow{B_{\mathcal{O}}} \mathrm{coAlg}^{\mathrm{dp}}_{\mathrm{Bar}(\mathcal{O})}(\mathcal{C}) \xrightarrow{\mathrm{forget}} \mathcal{C}.
\end{equation*}
In other words, for any $\mathcal{O}$-algebra $X$ the object $\mathrm{TAQ}_{\mathcal{O}}(X)$ naturally admits the structure of a conilpotent divided power $\mathrm{Bar}(\mathcal{O})$-coalgebra. Moreover, the functor $B_{\mathcal{O}}$ admits a right adjoint $C_{\mathrm{Bar}(\mathcal{O})}$.
\end{theorem}
\begin{proof}[Sketch of proof]
As explained in Section \ref{sec:monads}, the fact that $\mathrm{TAQ}_{\mathcal{O}}$ is left adjoint formally implies that it factors through the $\infty$-category of coalgebras for the comonad $\mathrm{TAQ}_{\mathcal{O}} \circ \mathrm{triv}_{\mathcal{O}}$. Thus it suffices to identify this comonad with $F_{\mathrm{Bar}(\mathcal{O})}$. To do this, recall that we may resolve any $\mathcal{O}$-algebra $X$ by free algebras and get an equivalence
\begin{equation*}
|\mathrm{Bar}(F_{\mathcal{O}}, F_{\mathcal{O}}, X)| \simeq X.
\end{equation*}
Hence 
\begin{eqnarray*}
\mathrm{TAQ}_{\mathcal{O}}(\mathrm{triv}_{\mathcal{O}} X) & \simeq & \mathrm{TAQ}_{\mathcal{O}}|\mathrm{Bar}(F_{\mathcal{O}}, F_{\mathcal{O}}, \mathrm{triv}_{\mathcal{O}} X)| \\
& \simeq & |\mathrm{Bar}(\mathrm{id}_{\mathcal{C}}, F_{\mathcal{O}}, \mathrm{id}_{\mathcal{C}})|(X) \\
& \simeq & F_{\mathrm{Bar}(\mathcal{O})}(X).
\end{eqnarray*}
The existence of the right adjoint $C_{\mathrm{Bar}(\mathcal{O})}$ can be deduced from the adjoint functor theorem. Alternatively, it may be constructed explicitly as the \emph{derived primitives} of a $\mathrm{Bar}(\mathcal{O})$-coalgebra, which is a construction formally dual to that of the derived indecomposables functor $\mathrm{TAQ}_{\mathcal{O}}$.
\end{proof}

The real challenge in understanding Koszul duality lies in identifying appropriate subcategories of $\mathrm{Alg}_{\mathcal{O}}(\mathcal{C})$ and $\mathrm{coAlg}^{\mathrm{dp}}_{\mathrm{Bar}(\mathcal{O})}(\mathcal{C})$ on which the functors $B_{\mathcal{O}}$ and $C_{\mathrm{Bar}(\mathcal{O})}$ give an equivalence of $\infty$-categories. Francis and Gaitsgory \cite{francisgaitsgory} describe several cases and generally conjecture that it should suffice to restrict to what are called pro-nilpotent algebras and ind-conilpotent coalgebras. A partial result goes as follows:

\begin{theorem}[Ching--Harper \cite{chingharper}]
\label{thm:chingharper}
Let $R$ be a commutative ring spectrum and $\mathcal{O}$ an operad in the $\infty$-category $\mathrm{Mod}_R$ of $R$-module spectra with $\mathcal{O}(1) = R$. Assume that $R$ and the terms $\mathcal{O}(n)$ are connective. Then the restriction 
\[
\begin{tikzcd}
\mathrm{Alg}_{\mathcal{O}}(\mathrm{Mod}_R)^{\geq 1} \ar[shift left]{r}{B_{\mathcal{O}}} & \mathrm{coAlg}_{\mathrm{Bar}(\mathcal{O})}(\mathrm{Mod}_R)^{\geq 1} \ar[shift left]{l}{C_{\mathrm{Bar}(\mathcal{O})}}
\end{tikzcd}
\]
of the adjunction of Theorem \ref{thm:koszul1} to connected objects is an equivalence of $\infty$-categories.
\end{theorem}

\begin{remark}
\label{rmk:tstructures}
In the presence of a useful notion of connectivity (such as a $t$-structure on $\mathcal{C}$), the strategy of proof of Theorem \ref{thm:chingharper} goes through much more generally.
\end{remark}

In particular, one may apply Theorem \ref{thm:chingharper} to the case where $R = H\mathbb{Q}$ and $\mathcal{O}$ is the Lie operad to retrieve Quillen's theorem \ref{thm:CE}.

\section{Spectral Lie algebras}
\label{sec:Lie}

\index{spectral Lie algebra}
In this section we describe an extension of the theory of Lie algebras to the $\infty$-category $\mathrm{Sp}$ of spectra. This extension has only recently begun to be exploited and promises to be very useful.

The non-unital commutative operad $\mathbf{Com}$ in the $\infty$-category of spectra is the operad parametrizing non-unital commutative ring spectra: it has $\mathbf{Com}(n) = \mathbb{S}$, the sphere spectrum, in every degree $n \geq 1$. The relevant structure maps 
\begin{equation*}
\mathbf{Com}(n) \otimes \mathbf{Com}(k_1) \otimes \cdots \otimes \mathbf{Com}(k_n) \rightarrow \mathbf{Com}(k_1 + \cdots + k_n)
\end{equation*}
are the canonical equivalences determined by the fact that $\mathbb{S}$ is the unit of the smash product. We can take the termwise Spanier--Whitehead dual $\mathbf{Com}^{\vee}$ to obtain the non-unital commutative cooperad, which of course still has every term equal to the sphere spectrum $\mathbb{S}$. The following definition is inspired by the duality between the Lie operad and the commutative cooperad in chain complexes described in Example \ref{ex:operads}:

\index{cobar construction}
\index{spectral Lie operad}
\begin{definition}
The \emph{spectral Lie operad} $\mathbf{L}$ is the cobar construction $\mathrm{Cobar}(\mathbf{Com}^{\vee})$. We write $\mathrm{Lie}(\mathrm{Sp})$ for the $\infty$-category of $\mathbf{L}$-algebras in $\mathrm{Sp}$ and refer to its objects as \emph{spectral Lie algebras}. 
\end{definition}

\begin{remark}
The commutative cooperad can be constructed in any symmetric monoidal $\infty$-category $\mathcal{C}$, since it only uses the unit of the symmetric monoidal structure. Hence one can make the above definition of Lie algebras in essentially any context, although one should usually require $\mathcal{C}$ to be stable for this to be useful.
\end{remark}

The operad $\mathbf{L}$ was first constructed by Salvatore \cite{salvatore} and Ching \cite{ching}. The spectra $\mathbf{L}(n)$ can be described very explicitly as the Spanier--Whitehead duals of certain finite simplicial complexes, which we will now explain. These descriptions originate in the work of Johnson on the Goodwillie derivatives of the identity \cite{johnson} (on which we will have more to say below) and were reformulated in the form we present below by Arone--Mahowald \cite{aronemahowald}. Write $\mathbf{P}(n)$ for the set of partitions of (meaning equivalence relations on) the set $\{1, \ldots, n\}$. We regard $\mathbf{P}(n)$ as a partially ordered set under refinement of partitions. It has a minimal (resp. maximal) element, namely the trivial partition with only one equivalence class (resp. the discrete partition). We write $\mathbf{P}^+(n)$ (resp. $\mathbf{P}^-(n)$) for the subset of $\mathbf{P}(n)$ obtained by discarding the minimal element (resp. the maximal element). Also, we write $\mathbf{P}^{\pm}(n)$ for the intersection $\mathbf{P}^+(n) \cap \mathbf{P}^-(n)$. Note that the symmetric group $\Sigma_n$ naturally acts on $\mathbf{P}(n)$, as well as on the various subsets we have defined. 

\index{partition complex}
\begin{definition}
The \emph{$n$th partition complex} $\Pi_n$ is the $\Sigma_n$-space $|\mathbf{P}^{\pm}(n)|$ obtained as the geometric realization of the nerve of the poset $\mathbf{P}^{\pm}$. Furthermore, we define
\begin{equation*}
K_n := |\mathbf{P}(n)|/(|\mathbf{P}^+(n)| \cup |\mathbf{P}^-(n)|).
\end{equation*}
\end{definition}

The space $K_n$ is homotopy equivalent to the double suspension $\Sigma S\Pi_n$ of the partition complex. Here $S$ denotes unreduced suspension, whereas $\Sigma$ denotes reduced suspension, with $S\Pi_n$ regarded as pointed at one of the two cone points.

\begin{example}
The space $\Pi_2$ is empty, so that $K_2$ is $S^1$ with trivial $\Sigma_2$-action. The space $\Pi_3$ is discrete and has three points, corresponding to the partition $(12)(3)$ and its permutations. As a $\Sigma_3$-space it is isomorphic to $\Sigma_3/\Sigma_2$. Hence $K_3$ is homotopy equivalent to a wedge of two 2-spheres, although this identification disregards the $\Sigma_3$-action. The space $\Pi_4$ is a one-dimensional simplicial complex homotopy equivalent to a wedge of six circles, so that $K_4$ is equivalent to a wedge of six 3-spheres. Again one should be careful that this is an identification of the homotopy type of the underlying space, disregarding the action of $\Sigma_4$.
\end{example}

The connection between the cohomology of partition complexes and Lie algebras goes back to work of Hanlon \cite{hanlon}, Stanley \cite{stanley}, Joyal \cite{joyallie}, and Barcelo \cite{barcelo}. Ching \cite{ching} established a topological refinement of this connection; to be precise, he observed that the terms of the cobar construction of the commutative cooperad are precisely the Spanier--Whitehead duals of the $K_n$:
\begin{equation*}
\mathbf{L}(n) \cong (\Sigma^\infty K_n)^{\vee}.
\end{equation*}
The operad structure of $\mathbf{L}$ is reflected in a cooperad structure on the collection of spaces $\{K_n\}_{\geq 1}$. Roughly speaking it can be described by observing that $K_n$ is homeomorphic to a certain space of weighted rooted trees with $n$ leaves, with comultiplication defined by decomposing trees into smaller subtrees grafted along a common edge. In general the space $K_n$ is homotopy equivalent to a wedge of $(n-1)!$ spheres of dimension $n-1$. Detailed results on the equivariant topology of the partition complexes can be found in the works of Arone--Dwyer \cite{aronedwyer} and Arone--Brantner \cite{aronebrantner}. 

Write $\mathbf{Lie}$ for the usual Lie operad in abelian groups. Taking the integral homology of the spectra $\mathbf{L}(n)$ gives an operad in graded abelian groups, which is precisely a degree shift of $\mathbf{Lie}$ thought of as sitting in homological degree zero. Indeed, as graded abelian groups one has
\begin{equation*}
H_*\mathbf{L}(n) \cong \mathbf{Lie}(n)[1-n],
\end{equation*}
but it is better to write
\begin{equation*}
H_*\mathbf{L}(n) \cong (\mathbf{Lie}(n)[1]) \otimes (\mathbb{Z}[-1])^{\otimes n}
\end{equation*}
to make the action of $\Sigma_n$ explicit. It acts by permuting the $n$ factors of $\mathbb{Z}[-1]$ on the right-hand side; said differently, the $\Sigma_n$-action on $H_*\mathbf{L}(n)$ is the same as the action on $\mathbf{Lie}(n)$ twisted by the sign representation. These identifications are compatible with the operad structures on both sides. One way to prove these facts is to relate the homology of the cobar construction that defines $\mathbf{L}$ to the algebraic cobar construction for the commutative cooperad in graded abelian groups.

\begin{remark}
This shifting of degree for an operad is a rather harmless procedure. Indeed, endowing a graded abelian group $M$ with the structure of an algebra for the operad $H_*\mathbf{L}$ is the same thing as giving $M[-1]$ the structure of a graded Lie algebra. The occurrence of these shifts also explains the degree shift occurring in Theorem \ref{thm:CE}.
\end{remark}

\index{Goodwillie tower of the identity}
A direct connection between the spectral Lie operad and the homotopy theory of spaces (and indeed, one of the original motivations for studying this operad) is given by the \emph{Goodwillie derivatives of the identity}. We refer to the chapter of Arone and Ching in this same volume for a survey of Goodwillie calculus, but let us summarize what is relevant for us here. The Goodwillie tower of the identity functor on the $\infty$-category $\mathcal{S}_*$ gives, for each pointed space $X$, a tower of spaces interpolating between $X$ and its stable homotopy type:
\[
\begin{tikzcd}
&& \vdots \ar{d} \\
&& P_3 X \ar{d} \\
&& P_2 X \ar{d} \\
X \ar{rr}\ar{urr}\ar{uurr} && P_1 X = \Omega^\infty\Sigma^\infty X.
\end{tikzcd}
\]
The fiber $D_n X$ of the map
\begin{equation*}
P_n X \rightarrow P_{n-1} X
\end{equation*}
is called the \emph{$n$th homogenous layer} of the tower. It turns out to be an infinite loop space $D_n X = \Omega^\infty \mathbf{D}_n X$ associated with a spectrum $\mathbf{D}_n X$ which can be written as 
\begin{equation*}
\mathbf{D}_n X \simeq (\partial_n \mathrm{id} \otimes \Sigma^\infty X^{\otimes n})_{h\Sigma_n}
\end{equation*}
for some spectrum $\partial_n \mathrm{id}$ with $\Sigma_n$-action called the \emph{$n$th derivative} of the identity functor. 

The connection to spectral Lie algebras is that $\partial_n \mathrm{id}$ can be identified with $\mathbf{L}(n)$. The homotopy type of $\partial_n \mathrm{id}$ was first determined by Johnson \cite{johnson}. We follow a line of reasoning due to Arone--Ching \cite{aroneching}, because it relates directly to our earlier discussion of cobar constructions. They prove that the map
\begin{equation*}
\mathrm{id} \rightarrow \mathrm{Tot}\bigl(\mathrm{Cobar}(\Omega^\infty, \Sigma^\infty\Omega^\infty, \Sigma^\infty)^{\bullet}\bigr)
\end{equation*}
arising from the cosimplicial resolution of the identity functor via $\Omega^\infty\Sigma^\infty$ (see Section \ref{sec:monads}) induces an equivalence of derivatives
\begin{equation*}
\partial_*\mathrm{id} \simeq \mathrm{Tot}\bigl(\mathrm{Cobar}(\partial_*\Omega^\infty, \partial_*(\Sigma^\infty\Omega^\infty), \partial_*\Sigma^\infty)^\bullet\bigr).
\end{equation*}
To interpret the right-hand side, one uses that the functors $\Omega^\infty$ and $\Sigma^\infty$ are linear (in Goodwillie's sense), that the derivatives of $\Sigma^\infty\Omega^\infty$ can be identified with the commutative cooperad $\mathbf{Com}^{\vee}$ (with the cooperad structure corresponding to the fact that this functor is a comonad), and that the chain rule for functors from the $\infty$-category $\mathrm{Sp}$ to itself states
\begin{equation*}
\partial_*(GF) \simeq \partial_*G \circ \partial_*F.
\end{equation*}
The right-hand side denotes the composition product of symmetric sequences. Putting these ingredients together gives
\begin{equation*}
\partial_*\mathrm{id} \simeq \mathrm{Tot}\bigl(\mathrm{Cobar}(\mathbf{1}, \mathbf{Com}^{\vee}, \mathbf{1})^\bullet\bigr) = \mathrm{Cobar}(\mathbf{Com}^{\vee})
\end{equation*}
and the right-hand side is precisely our definition of $\mathbf{L}$. Thus, the Goodwillie tower produces a spectral sequence converging to the homotopy groups of $X$, starting from the homotopy groups of the spectrum
\begin{equation*}
\bigoplus_{n \geq 1} (\mathbf{L}(n) \otimes \Sigma^\infty X^{\otimes n})_{h\Sigma_n},
\end{equation*}
which is the \emph{free} spectral Lie algebra on the suspension spectrum of $X$. This \emph{Goodwillie spectral sequence} has been used very succesfully by Behrens \cite{behrensEHP}, who combines it with the EHP sequence to reproduce a significant part of Toda's calculations of unstable homotopy groups of spheres.

A starting point for most calculations with the Goodwillie tower is the homology of the spectra $\mathbf{D}_n X$, which has been studied in great detail by Arone--Mahowald \cite{aronemahowald} and Arone--Dwyer \cite{aronedwyer} in the case where $X$ is a sphere. For example, when this sphere is of odd dimension then the layers $\mathbf{D}_n X$ are contractible whenever $n$ is not a power of a prime. When $n=2^k$, its cohomology with coefficients in $\mathbb{F}_2$ is free over the subalgebra $\mathcal{A}^*_{k-1}$ of the Steenrod algebra $\mathcal{A}^*$ generated by $\mathrm{Sq}^1, \ldots, \mathrm{Sq}^{2^{k-1}}$ (and a similar statement holds at odd primes). This has very useful consequences for the analysis of the $v_n$-periodic homotopy groups of the Goodwillie tower. Also, the spaces $\Pi_{p^k}$ are closely related to Tits buildings for the groups $\mathrm{GL}_k(\mathbb{F}_p)$, and using this Arone--Dwyer relate the calculation of the mod $p$ homology of $\mathbf{D}_{p^k}(X)$ to that of $\mathrm{GL}_k(\mathbb{F}_p)$ with coefficients in the Steinberg module. The reader can find a much more elaborate discussion of these results in the chapter of Arone--Ching in this same volume. These homology calculations also lead to a theory of power operations for spectral Lie algebras; we refer to Behrens \cite{behrensEHP} and Antol\'{i}n-Camarena \cite{antolincamarena} for the case of $\mathbb{F}_2$-coefficients, Kjaer \cite{kjaer} for $\mathbb{F}_p$-coefficients with $p > 2$, and Brantner \cite{brantner} for the case of Morava $E$-theory (in particular including $p$-complete complex $K$-theory).

\section{Periodic unstable homotopy theory}
\label{sec:periodicity}

\index{chromatic homotopy theory}
In this section we discuss some fundamental concepts of chromatic homotopy theory and emphasize their role in unstable homotopy theory. The reader can find a much more thorough exposition of the chromatic perspective on stable homotopy theory in the chapter of Barthel and Beaudry in this volume or consult some of the standard references \cite{devinatzhopkinssmith,hopkins,hopkinssmith,ravenelgreen,ravenelorange}.

The rational homotopy groups of a pointed space $X$ are the result of considering the homotopy classes of maps $[S^k, X]_*$ from spheres to $X$ and subsequently inverting the action of the degree $p$ maps $p: S^k \rightarrow S^k$ for all primes $p$. As demonstrated by Serre's calculation of the rational homotopy groups of spheres and the rational homotopy theory of Quillen and Sullivan, these rational homotopy groups are surprisingly tractable invariants.

If one wants to move beyond rational homotopy theory, a natural starting point is the `mod $p$ homotopy groups' $[S^k/p, X]_* =: \pi_{k-1}(X; \mathbb{Z}/p)$, where $S^k/p$ denotes the cofiber of the degree $p$ map on $S^k$ (i.e. a mod $p$ Moore space). These are groups when $k \geq 2$, which are abelian if $k \geq 3$. The complexity of calculating them is on par with that of the usual homotopy groups. However, it turns out one can again invert the action of a certain map to make the problem more tractable. To be precise, there is a certain map
\begin{equation*}
\alpha: \Sigma^d S^k/p \rightarrow S^k/p
\end{equation*}
seemingly due to Barratt, described by Adams in \cite{adams}. Here for $p$ odd one has $d = 2(p-1)$ and $k \geq 3$, and for $p=2$ one should take $d=8$ and $k \geq 5$. The crucial feature of this map is that it induces an isomorphism on complex $K$-theory. In particular, any number of iterates of (suspensions of) $\alpha$ is \emph{not} null-homotopic. In fact, $\alpha$ induces multiplication by the $(p-1)$st power of the Bott class (for $p$ odd) or its fourth power (for $p=2$). As such one can think of it as a geometric manifestation of Bott periodicity. The map $\alpha$ provides an action of the graded ring $\mathbb{Z}[\alpha]$ (with $|\alpha| = d$) on the graded abelian group $\pi_*(X; \mathbb{Z}/p)$ (with $\ast \geq 2$) and one defines the \emph{$v_1$-periodic mod $p$ homotopy groups} of $X$ to be
\begin{equation*}
v_1^{-1}\pi_*(X; \mathbb{Z}/p) := \mathbb{Z}[\alpha^{\pm 1}] \otimes_{\mathbb{Z}[\alpha]} \pi_*(X; \mathbb{Z}/p).
\end{equation*}
The rational and $v_1$-periodic homotopy groups of spaces form the beginning of a hierarchy of \emph{$v_n$-periodic homotopy groups}, which we will define shortly. This hierarchy is closely related to the sequence of `prime' localizations $L_{H\mathbb{Q}}$, $L_{K(1)}$, $L_{K(2)}$, $\ldots$, of stable homotopy theory at the Morava $K$-theories, which play the role of prime fields in the $\infty$-category of spectra. The latter picture is discussed in detail in the chapter of Barthel--Beaudry in this volume. While the focus for them is mostly on these chromatic localizations (and the closely related localizations $L_n$), for us the fundamental ingredient of chromatic homotopy theory will be the periodicity results of Hopkins and Smith \cite{hopkinssmith} (and the finite localizations $L_n^f$). We begin with a brief recollection of the thick subcategory theorem.

\index{Morava $K$-theory}
From now on we fix a prime $p$ and work in the $\infty$-categories of $p$-local spectra and spaces. We will often leave the adjective $p$-local implicit. We say a spectrum $X$ is \emph{of type $\geq n$} if $K(i)_*X = 0$ for $i < n$, and we say $X$ is \emph{of type $n$} if it is of type $\geq n$ and additionally $K(n)_*X \neq 0$. It turns out that if a finite spectrum $X$ has $K(n-1)_*X = 0$ then also $K(i)_*X = 0$ for $i < n - 1$, so that the former is a sufficient condition for $X$ to be of type $\geq n$. Write $\mathrm{Sp}^{\mathrm{fin}}_{(p)}$ for the $\infty$-category of $p$-local finite spectra and $\mathrm{Sp}^{\mathrm{fin}}_{\geq n}$ for the subcategory of those $X$ which are of type $\geq n$. This subcategory is \emph{thick}: if two terms in a cofiber sequence are contained in it, then so is the third, and moreover it is closed under retracts. These thick subcategories form a nested sequence
\begin{equation*}
\cdots \subseteq \mathrm{Sp}^{\mathrm{fin}}_{\geq n+1} \subseteq \mathrm{Sp}^{\mathrm{fin}}_{\geq n} \subseteq \cdots \subseteq  \mathrm{Sp}^{\mathrm{fin}}_{\geq 0}  = \mathrm{Sp}^{\mathrm{fin}}_{(p)}.
\end{equation*}
The fact that these are proper inclusions is a highly nontrivial result of Mitchell \cite{mitchell}: for every $n$ there exists a finite spectrum of type $n$. The following explains the fundamental importance of this notion for stable homotopy theory:
\index{thick subcategory}

\begin{theorem}[The thick subcategory theorem, Hopkins--Smith \cite{hopkinssmith}]
Every thick subcategory of $\mathrm{Sp}^{\mathrm{fin}}_{(p)}$ is of the form $\mathrm{Sp}^{\mathrm{fin}}_{\geq n}$ for some $n$.
\end{theorem}

The filtration of the $\infty$-category of finite $p$-local spectra above also gives a filtration
\begin{equation*}
\cdots \subseteq \mathrm{Sp}_{\geq n+1} \subseteq \mathrm{Sp}_{\geq n} \subseteq \cdots \subseteq \mathrm{Sp}_{\geq 0}  = \mathrm{Sp}_{(p)}
\end{equation*}
of the $\infty$-category of all (not necessarily finite) $p$-local spectra by localizing subcategories. Here we write $\mathrm{Sp}_{\geq n}$ for the smallest subcategory of $\mathrm{Sp}$ which contains $\mathrm{Sp}^{\mathrm{fin}}_{\geq n}$ and is closed under colimits. As with any filtration, the goal is now to understand its `associated graded' and subsequently investigate how the layers fit together. In this survey we will be almost exclusively concerned with the first aspect.

To define what we mean by associated graded, first observe that the tower of subcategories above gives a corresponding tower of localizations
\begin{equation*}
\cdots \rightarrow L_n^f \mathrm{Sp} \rightarrow L_{n-1}^f \mathrm{Sp} \rightarrow \cdots \rightarrow L_0^f \mathrm{Sp} = \mathrm{Sp}_{\mathbb{Q}}.
\end{equation*}
Here $L_n^f \mathrm{Sp}$ is the localization (in the sense of Definition \ref{def:localization}) of $\mathrm{Sp}_{(p)}$ at the set of maps $f\colon X \rightarrow Y$ whose cofiber is contained in $\mathrm{Sp}_{\geq n+1}$. Said differently, it is the quotient $\mathrm{Sp}_{(p)}/\mathrm{Sp}_{\geq n+1}$ computed in the $\infty$-category of stable $\infty$-categories and exact functors. Since $\mathrm{Sp}_{\geq n+1}$ is closed under colimits in $\mathrm{Sp}_{(p)}$, all these localizations are reflective. We write $L_n^f\colon \mathrm{Sp}_{(p)} \rightarrow \mathrm{Sp}_{(p)}$ for the composite of the localization functor $\mathrm{Sp}_{(p)} \rightarrow L_n^f \mathrm{Sp}$ and its right adjoint. With this notation we have natural transformations $L_n^f \rightarrow L_{n-1}^f$. As a consequence of the thick subcategory theorem, one can in fact characterize $L_n^f$ as the localization functor which kills a \emph{single} finite type $n+1$ spectrum $V$. The layers of our filtration can now be described as follows:

\index{$v_n$-periodic}
\begin{definition}
The $\infty$-category $\mathrm{Sp}_{v_n}$ of \emph{$v_n$-periodic spectra} is the quotient $\mathrm{Sp}_{\geq n}/\mathrm{Sp}_{\geq n+1}$, meaning it is the universal stable $\infty$-category which receives an exact functor from $\mathrm{Sp}_{\geq n}$ which is identically zero on the subcategory $\mathrm{Sp}_{\geq n+1}$.
\end{definition}

\index{$v_n$ self-map}
This definition might seem abstract, but we will recast it in more concrete (and perhaps more familiar) terms using another fundamental result of Hopkins and Smith. To state it we need some terminology. A \emph{$v_n$ self-map} of a $p$-local spectrum $X$ is a map $v\colon \Sigma^d X \rightarrow X$, for some integer $d \geq 0$, with the property that $K(n)_*v$ is an isomorphism and $K(m)_*v$ is nilpotent whenever $m \neq n$. If $n=0$ one always has $d=0$ and one should require that $v$ acts by multiplication by a rational number on $K(0)_*X = H_*(X;\mathbb{Q})$. For general $n\geq 1$, replacing $v$ by a sufficiently high power if necessary one can always arrange that on $K(n)_*X$ it acts by a power of $v_n \in K(n)_*$ and by zero on $K(m)_*X$ for $m \neq n$. The typical examples to keep in mind are the following: multiplication by $p$ is a $v_0$ self-map (and exists for any spectrum), whereas the Adams map $\alpha$ is a $v_1$ self-map of the mod $p$ Moore spectrum, which is of type 1.

\begin{theorem}[The periodicity theorem, Hopkins--Smith \cite{hopkinssmith}]
\label{thm:periodicity}
If $X$ is a finite spectrum of type $\geq n$, then it admits a $v_n$ self-map. Furthermore, $v_n$ self-maps are asymptotically unique in the sense that for $f\colon X \rightarrow Y$ a map of finite spectra and $v_n$ self-maps $v,w$ of $X,Y$ respectively, there exist $M,N \gg 0$ for which the square
\[
\begin{tikzcd}
\Sigma^{Md} X \ar{d}{v^N} \ar{r}{f} & \Sigma^{Ne} Y \ar{d}{w^M} \\
X \ar{r}{f} & Y
\end{tikzcd}
\]
commutes up to homotopy, where $d$ and $e$ are the degrees of $v$ and $w$ respectively. In particular, taking $f$ to be the identity, any two $v_n$ self-maps on $X$ become homotopic after sufficiently many iterations.
\end{theorem}

Now suppose $V$ is a finite spectrum of type $n$ and $v\colon \Sigma^d V \rightarrow V$ is a $v_n$ self-map. For any spectrum $X$ we can define its $v$-periodic homotopy groups by
\begin{equation*}
v^{-1}\pi_*(X; V) := \mathbb{Z}[v^{\pm 1}] \otimes_{\mathbb{Z}[v]}\pi_*\mathrm{Map}(V,X).
\end{equation*}
An equivalent description is as follows. Since $V$ is finite, the mapping spectrum $F(V,X)$ is equivalent to the smash product $V^{\vee} \otimes X$, with $V^{\vee}$ the Spanier--Whitehead dual of $V$. Form the \emph{telescope}
\begin{equation*}
v^{-1}V^{\vee} := \varinjlim(V^{\vee} \xrightarrow{v^{\vee}} \Sigma^{-d}V^{\vee} \xrightarrow{v^{\vee}} \cdots).
\end{equation*}
Then $v^{-1}\pi_*(X; V) \cong \pi_*(v^{-1}V^{\vee} \otimes X)$.

\index{$v_n$-periodic equivalence}
\begin{definition}
A map of spectra is a \emph{$v_n$-periodic equivalence} if it induces an isomorphism on $v$-periodic homotopy groups.
\end{definition}

This definition only depends on $n$; indeed, the thick subcategory theorem implies that it is independent of the choice of $V$, whereas the asymptotic uniqueness of $v_n$ self-maps gives independence of the choice of $v$. It is common to write $T(n) = v^{-1}V^{\vee}$, so that a $v_n$-periodic equivalence is by definition a $T(n)_*$-equivalence. The notation is a little ambiguous, because $T(n)$ depends on choices. However, the associated notion of equivalence does not. The basic facts to keep in mind are the following:
\begin{itemize}
\item[(a)] As a consequence of the periodicity theorem, any $v_n$ self-map of a finite spectrum of type $\geq n$ is a $v_n$-periodic equivalence.
\item[(b)] If $W$ is a finite spectrum of type $\geq n+1$, then $W \rightarrow 0$ is a $v_n$-periodic equivalence. This is immediate from (a) and the fact that the null map $W \xrightarrow{0} W$ is a $v_n$ self-map.
\end{itemize}

We now characterize $\mathrm{Sp}_{v_n}$ in terms of the $v_n$-periodic equivalences and describe two ways in which it can be realized as a full subcategory of $\mathrm{Sp}_{(p)}$:

\begin{proposition}
\label{prop:MnfvsTn}
The $\infty$-category $\mathrm{Sp}_{v_n}$ of $v_n$-periodic spectra is the localization (in the sense of Definition \ref{def:localization}) of $\mathrm{Sp}_{(p)}$ at the $v_n$-periodic equivalences. It is equivalent to the following two subcategories of $\mathrm{Sp}_{(p)}$:
\begin{itemize}
\item[(1)] The full subcategory $\mathrm{Sp}_{T(n)}$ of $T(n)$-local spectra in the sense of Bousfield.
\item[(2)] The full subcategory $M_n^f\mathrm{Sp}$ of spectra of the form $M_n^f X$, where $M_n^f$ denotes the fiber of the natural transformation $L_n^f \rightarrow L_{n-1}^f$.
\end{itemize}
Furthermore, the functors
\[
\begin{tikzcd}
M_n^f\mathrm{Sp} \ar[shift left]{r}{L_{T(n)}} & \mathrm{Sp}_{T(n)} \ar[shift left]{l}{M_n^f}
\end{tikzcd}
\]
are mutually inverse equivalences.
\end{proposition}
\begin{proof}[Sketch of proof]
The $\infty$-category $\mathrm{Sp}_{T(n)}$ is the localization at the $v_n$-periodic equivalences (which equal $T(n)_*$-equivalences) by definition. The equivalence between $\mathrm{Sp}_{T(n)}$ and $M_n^f\mathrm{Sp}$ follows from the fact that all of the maps in the diagram
\[
\begin{tikzcd}
& X \ar{d} & \\
M_n^f X \ar{r} & L_n^f X \ar{r} & L_{T(n)} X
\end{tikzcd}
\]
are $v_n$-periodic equivalences. Finally, $\mathrm{Sp}_{v_n}$ is equivalent to $M_n^f\mathrm{Sp}$. Indeed, by construction $\mathrm{Sp}_{v_n}$ is the essential image of $\mathrm{Sp}_{\geq n}$ in $L_n^f\mathrm{Sp}$, which we identify with the full subcategory of $\mathrm{Sp}_{(p)}$ consisting of $L_n^f$-local spectra. This essential image is contained in $M_n^f\mathrm{Sp}$, because for $V \in \mathrm{Sp}_{\geq n}$ the spectrum $L_{n-1}^f V$ is null. It is also all of $M_n^f\mathrm{Sp}$, since the fiber of the map $X \rightarrow L_{n-1}^f X$ is a colimit of spectra of type $\geq n$.
\end{proof}

\begin{remark}
The spectra $M_n^f X$ are often called the \emph{monochromatic} or \emph{monocular} layers of $X$ (the latter term is used by Bousfield \cite{bousfieldtelescopic}, who attributes it to Ravenel). The equivalence between $M_n^f\mathrm{Sp}$ and $\mathrm{Sp}_{T(n)}$ is analogous to the equivalence between $p$-primary torsion and (derived) $p$-complete objects of the derived category $D(\mathbb{Z})$.
\end{remark}


We now proceed to the unstable case. The theory of localizations of unstable homotopy theory was developed by Bousfield \cite{bousfieldlocalizationperiodicity,bousfieldlocalizations,bousfieldtelescopic} and Dror-Farjoun \cite{drorfarjoun}. More details on the material we discuss below can also be found in \cite{heuts} and \cite{thursday}.

A pointed space $V$ is of type $\geq n$ precisely if its suspension spectrum is. Similarly, a $v_n$ self-map of such a space is a self-map $v$ such that $\Sigma^\infty v$ is a $v_n$ self-map of the suspension spectrum $\Sigma^\infty V$. Like before, one now defines the $v$-periodic homotopy groups of another pointed space $X$ by inverting the action of $v$ on the homotopy groups of the mapping space $\mathrm{Map}_*(V,X)$. A $v_n$-periodic equivalence of spaces is a map which induces isomorphisms on these periodic homotopy groups. Again we would like to formally invert the $v_n$-periodic equivalences of pointed spaces. One can deduce from results of Bousfield that this is possible:

\begin{theorem}[See \cite{bousfieldtelescopic} and \cite{heuts}]
\label{thm:Mnf}
The localization of $\mathcal{S}_*$ at the $v_n$-periodic equivalences (in the sense of Definition \ref{def:localization}) exists; we denote it by $M\colon \mathcal{S}_* \rightarrow \mathcal{S}_{v_n}$. It has the following properties:
\begin{itemize} 
\item[(1)] The functor $M$ preserves finite limits and filtered colimits.
\item[(2)] The $\infty$-category $\mathcal{S}_{v_n}$ is compactly generated. If $V$ is any pointed finite type space of type $n$, then $M(\Sigma V)$ is a compact object and generates $\mathcal{S}_{v_n}$ under colimits.
\item[(3)] The stabilization of $\mathcal{S}_{v_n}$ is equivalent to $\mathrm{Sp}_{v_n}$.
\end{itemize}
\end{theorem}

The reader should be warned that the localization $M$ is \emph{not} reflective; in particular, it does not arise directly from a left Bousfield localization on the level of model categories. Rather, one should think of it as the composition of a left and then a right localization, as will become apparent. Theorem \ref{thm:Mnf} is proved by explicitly constructing an $\infty$-category analogous to $M_n^f\mathrm{Sp}$ in the stable case. A crucial ingredient is Bousfield's classification of the localizations associated to `nullifying' a finite space. If $A$ is any space, one calls a space $X$ \emph{$P_A$-local} (or \emph{$A$-null}) if the evident map
\begin{equation*}
X \simeq \mathrm{Map}(*, X) \rightarrow \mathrm{Map}(A,X)
\end{equation*}
is a homotopy equivalence. 
Essentially, the space $A$ is contractible from the point of view of $X$. A general space $X$ admits a $P_A$-localization $X \rightarrow P_A X$, and the subcategory of $\mathcal{S}_*$ on pointed $P_A$-local spaces is precisely the localization of $\mathcal{S}_*$ with respect to the map $A \rightarrow *$. We write $\langle A \rangle$ for the collection of spaces $Y$ for which $P_A Y \simeq *$ and call it the \emph{Bousfield class} of $A$. These are partially ordered; we write $\langle A \rangle \leq \langle B \rangle$ if every $Y$ for which $P_A Y \simeq *$ also satisfies $P_B Y \simeq *$. One can regard the following as an unstable analog of the thick subcategory theorem. Its proof relies on the stable version.

\begin{theorem}[Bousfield {\cite[Theorem 9.15]{bousfieldlocalizationperiodicity}}]
\label{thm:Bousfieldclasses}
Let $W$ and $W'$ be $p$-local finite pointed spaces that are also suspensions. Then the following are equivalent:
\begin{itemize}
\item[(1)] $\langle W \rangle \leq \langle W' \rangle$,
\item[(2)] $\mathrm{type}(W) \geq \mathrm{type}(W')$ and $\mathrm{conn}(W) \geq \mathrm{conn}(W')$, with $\mathrm{conn}(W)$ denoting the minimal $i$ with $\pi_i W \neq 0$. 
\end{itemize}
\end{theorem}

Thus, up to keeping track of connectivity, finite suspension spaces are still classified by their type, as was the case with finite spectra. Now pick a finite suspension $V_{n+1}$ of type $n+1$. We will denote $\mathrm{conn}(V_{n+1})$ by $d_{n+1}$ and the localization functor $P_{V_{n+1}}$ by $L_n^f$. The latter is of course slightly abusive, since $L_n^f$ depends not only on $n$, but also on the connectivity $d_{n+1}$ (but no more, according to Theorem \ref{thm:Bousfieldclasses}). More importantly, Theorem \ref{thm:Bousfieldclasses} immediately implies that for any pointed space $X$, the $v_i$-periodic homotopy groups of $L_n^f X$ vanish for $i > n$. Bousfield also proves that any pointed space $X$, the map $X \rightarrow L_n^f X$ is a $v_i$-periodic equivalence for $i \leq n$ (see \cite{bousfieldtelescopic}). In the stable case this implication can be reversed. Unstably one has to take a little care, because $L_n^f$ does not affect the homotopy groups of $X$ in dimensions below $d_{n+1}$. However, the following is true:

\begin{proposition}[Bousfield \cite{bousfieldtelescopic}]
A map $\varphi\colon X \rightarrow Y$ of $d_{n+1}$-connected pointed spaces is a $v_i$-periodic equivalence for every $0 \leq i \leq n$ if and only if $L_n^f(\varphi)$ is an equivalence of pointed spaces.
\end{proposition}

We define $M_n^f$ in analogy with the stable case as the fiber of the natural transformation $L_n^f \rightarrow L_{n-1}^f$; the latter exists by Theorem \ref{thm:Bousfieldclasses} as long as we have arranged our choices so that $\mathrm{conn}(V_n) \leq \mathrm{conn}(V_{n+1})$, which we will always assume to be the case. For $X$ a pointed space, there are maps
\begin{equation*}
X \rightarrow L_n^f X \leftarrow M_n^f X
\end{equation*}
that are both $v_n$-periodic equivalences. Moreover, for $i \neq n$ the $v_i$-periodic homotopy groups of $M_n^f X$ vanish. Write $\mathcal{M}_n^f$ for the full subcategory of $\mathcal{S}_*$ on spaces of the form $(M_n^f X)\langle d_{n+1}\rangle$, where the brackets indicate the $d_{n+1}$-connected cover. The functor
\begin{equation*}
X \mapsto (M_n^f X) \langle d_{n+1} \rangle
\end{equation*}
is naturally related to the identity by a zig-zag of $v_n$-periodic equivalences; moreover, a map between spaces in $\mathcal{M}_n^f$ is an equivalence if and only if it is a $v_n$-periodic equivalence. From this one deduces the first statement of Theorem \ref{thm:Mnf}. For the remainder of the proof we refer the reader to \cite{heuts}.


\index{Bousfield--Kuhn functor}
We conclude this section with a review of the Bousfield--Kuhn functor and its adjoint. We refer to \cite{bousfieldtelescopic} and \cite{kuhntelescopic} for a much more thorough discussion. Suppose $V$ is a finite pointed space of type $n$ with a $v_n$ self-map $v\colon \Sigma^d V \rightarrow V$. Then for $X$ a pointed space, one can define a (pre)spectrum $\Phi_v X$ with constituent spaces
\begin{equation*}
(\Phi_v X)_0 = \mathrm{Map}_*(V,X), \, (\Phi_v X)_d = \mathrm{Map}_*(V,X), \, \ldots, (\Phi_v X)_{kd} = \mathrm{Map}_*(V,X), \, \ldots
\end{equation*}
and structure maps
\begin{equation*}
(\Phi_v X)_{kd} = \mathrm{Map}_*(V,X) \xrightarrow{v^*} \mathrm{Map}_*(\Sigma^d V, X) \cong \Omega^d(\Phi_v X).
\end{equation*}
This defines the \emph{telescopic functor}
\begin{equation*}
\Phi_v: \mathcal{S}_* \rightarrow \mathrm{Sp}
\end{equation*}
associated to $v$. By construction it satisfies $\pi_*\Phi_v X \cong v^{-1}\pi_*(X;V)$. In fact this functor takes values in $T(n)$-local spectra, so that one may replace the codomain by $\mathrm{Sp}_{T(n)}$ or, equivalently, $\mathrm{Sp}_{v_n}$. As a consequence of Theorem \ref{thm:periodicity} the telescopic functor $\Phi_v$ does not really depend on the choice of $v$, but only on $V$. We will write $\Phi_V$ instead of $\Phi_v$. In fact, the dependence on $V$ can be made (contravariantly) functorial using Theorem \ref{thm:periodicity}. This can be used to conveniently package the various telescopic functors into one. The \emph{Bousfield--Kuhn functor} is a functor
\begin{equation*}
\Phi_n: \mathcal{S}_* \rightarrow \mathrm{Sp}_{v_n}
\end{equation*}
satisfying the following properties (see \cite{kuhntelescopic}):
\begin{itemize}
\item[(1)] There are equivalences
\begin{equation*}
V^\vee \otimes \Phi_n X \simeq \Phi_V X,
\end{equation*}
natural in $X$ and $V$.
\item[(2)] The composition $\Phi_n\Omega^\infty$ is naturally equivalent to the localization functor $L_{T(n)}$. In particular, the $T(n)$-localization of any spectrum $E$ depends only on its zeroth space $\Omega^\infty E$, regardless of its structure as an infinite loop space.
\end{itemize}
The functor $\Phi_n$ is even characterized up to equivalence by property (1). The construction of $\Phi_n$ and the proof of its basic properties rely on the following trick of Kuhn:

\begin{lemma}[Kuhn \cite{kuhnmorava}]
\label{lem:kuhn}
For any $n \geq 0$ there exists a directed system 
\begin{equation*}
F(1) \rightarrow F(2) \rightarrow F(3) \rightarrow \cdots
\end{equation*}
of finite spectra of type $n$ equipped with a map
\begin{equation*}
\varinjlim_k F(k) \rightarrow \mathbb{S}
\end{equation*}
which is a $T(m)_*$-equivalence for every $m \geq n$.
\end{lemma}
\begin{proof}[Sketch of proof]
In case $n=1$ this is familiar from algebra: the colimit $\varinjlim_k \mathbb{S}^{-1}/p^k$ is equivalent to the sphere spectrum after $p$-completion, in the same way that the derived $p$-completion of $\mathbb{Z}/p^\infty[-1]$ is $\mathbb{Z}_p$ in the $\infty$-category $D(\mathbb{Z})$. The general proof goes by induction on $n$: given a finite spectrum $V$ of type $n-1$, one picks a $v_{n-1}$ self-map $v$. The cofibers $\Sigma^{-1}V/v^k$ of the maps
\begin{equation*}
\Sigma^{-1} V \xrightarrow{v^k} \Sigma^{-1-kd} V
\end{equation*}
form a directed system with a map to $V$. The map from the colimit $\Sigma^{-1}V/v^\infty$ to $V$ is a $T(m)_*$-equivalence for $m \geq n$ because the telescope $v^{-1} V$ is $T(m)$-acyclic.
\end{proof}

To construct $\Phi_n$ one picks a system as in Lemma \ref{lem:kuhn} and defines $\Phi_n := \varprojlim_k \Phi_{F(k)}$. Property (1) is now fairly easily deduced from the identification
\begin{equation*}
V^{\vee} \otimes \Phi_{F(k)} \simeq F(k)^\vee \otimes \Phi_V.
\end{equation*}
Note also that Lemma \ref{lem:kuhn} implies that $\Phi_n$ is completely determined by the functors $F(k)^\vee \otimes \Phi_n$, which shows that (1) indeed characterizes $\Phi_n$ up to equivalence. This characterization also shows that the choice of $F(k)$ is inessential to the definition of $\Phi_n$. Property (2), striking as it may be, is quite easily proved as well. Again it suffices to check it for $\Phi_V$. We should argue that for any spectrum $E$ there is a natural equivalence $\Phi_V\Omega^\infty(E) \simeq L_{T(n)} V^{\vee} \otimes E$. Since $\Phi_V\Omega^\infty$ preserves $v_n$-periodic equivalences, we may without loss of generality assume that $E$ is already $T(n)$-local. But then the maps
\begin{equation*}
\mathrm{Map}_*(V, \Omega^\infty E) \xrightarrow{v^*} \Omega^d \mathrm{Map}_*(V, \Omega^\infty E)
\end{equation*} 
are equivalences and the formula we wrote down for $\Phi_V\Omega^\infty(E)$ is already an $\Omega$-spectrum. Even better, it is clearly the $\Omega$-spectrum $V^{\vee} \otimes E$.

Since $\Phi_n$ sends $v_n$-periodic equivalences to equivalences (essentially by construction), it factors through the localization $\mathcal{S}_{v_n}$. We still denote the resulting functor by
\begin{equation*}
\Phi_n\colon \mathcal{S}_{v_n} \rightarrow \mathrm{Sp}_{v_n}.
\end{equation*}
The following will be crucial:

\begin{proposition}[Bousfield \cite{bousfieldtelescopic}]
The functor $\Phi_n$ admits a left adjoint $\Theta_n\colon \mathrm{Sp}_{v_n} \rightarrow \mathcal{S}_{v_n}$.
\end{proposition}
\begin{proof}[Sketch of proof]
An explicit construction of $\Theta_n$ is described in \cite{bousfieldtelescopic} and also in \cite{kuhntelescopic}. To argue existence, roughly one can do the following. It suffices to show that for every spectrum $E$, the functor
\begin{equation*}
\mathcal{S}_{v_n} \rightarrow \mathcal{S}: X \mapsto \mathrm{Map}(E, \Phi_n X)
\end{equation*}
is corepresentable by some object $\Theta_n(E)$. Moreover, it suffices to consider a collection of objects $E$ which generate $\mathrm{Sp}_{v_n}$ under colimits, such as the finite spectra of type $n$. For such an $E$, the functor above takes the form $X \mapsto \Omega^\infty\Phi_E X$. Considering the definition of the telescopic functor $\Phi_E$, one sees that for $X \in \mathcal{S}_{v_n}$ the colimit defining this functor becomes eventually constant, showing that indeed our functor is corepresented by (the image in $\mathcal{S}_{v_n}$ of) a finite type $n$ space $E'$ for which $\Sigma^\infty E' \simeq E$ in $\mathrm{Sp}_{v_n}$.
\end{proof}

We now have two adjunctions relating the $\infty$-category $\mathcal{S}_{v_n}$ of $v_n$-periodic spaces to its stable counterpart $\mathrm{Sp}_{v_n}$, organized in the following diagram:
\[
\begin{tikzcd}
\mathrm{Sp}_{v_n} \ar[shift left]{r}{\Theta_n} & \mathcal{S}_{v_n} \ar[shift left]{r}{\Sigma^\infty_{v_n}} \ar[shift left]{l}{\Phi_n} & \mathrm{Sp}_{v_n}. \ar[shift left]{l}{\Omega^\infty_{v_n}}
\end{tikzcd}
\]
The adjunction on the right is the stabilization of $\mathcal{S}_{v_n}$, which we already mentioned in Theorem \ref{thm:Mnf}. Property (2) of the Bousfield--Kuhn functor in fact implies that the horizontal composites $\Phi_n\Omega^\infty_{v_n}$ and $\Sigma^\infty_{v_n}\Theta_n$ are both equivalent to the identity functor of $\mathrm{Sp}_{v_n}$. In this way the diagram above is very much analogous to the one at the end of Section \ref{sec:rationalhomotopy} and the one right above Theorem \ref{thm:basterramandell}.

\section{Lie algebras and $v_n$-periodic spaces}
\label{sec:Liealgebras}

\index{spectral Lie algebra}
The aim of this section is to outline a proof of Theorem \ref{thm:vnperiodicspaces}, relating the $\infty$-category $\mathcal{S}_{v_n}$ to the $\infty$-category of spectral Lie algebras. More details can be found in \cite{heuts}. We also discuss to what extent there is a `model' for $\mathcal{S}_{v_n}$ in terms of commutative coalgebras, Koszul dual to the Lie algebra model provided by Theorem \ref{thm:vnperiodicspaces}. This second model is closely related to recent work of Behrens--Rezk \cite{behrensrezk}, as we will explain.

The proof of Theorem \ref{thm:vnperiodicspaces} begins with the following, showing that the $\infty$-category $\mathcal{S}_{v_n}$ can indeed be expressed as some kind of algebras in $\mathrm{Sp}_{v_n}$. Throughout this section we will use the $\infty$-categories $\mathrm{Sp}_{v_n}$ and $\mathrm{Sp}_{T(n)}$ interchangeably.

\index{monadic adjunction}
\begin{theorem}[Eldred--Heuts--Mathew--Meier \cite{ehmm}]
\label{thm:ehmm}
The adjoint pair $(\Theta_n, \Phi_n)$ is monadic, i.e., the functor
\begin{equation*}
\varphi: \mathcal{S}_{v_n} \rightarrow \mathrm{Alg}_{\Phi_n\Theta_n}(\mathrm{Sp}_{v_n})
\end{equation*}
induced by $\Phi_n$ is an equivalence of $\infty$-categories.
\end{theorem}
\begin{proof}[Sketch of proof]
By Theorem \ref{thm:barrbeck} it suffices to check that $\Phi_n$ is conservative (which is essentially immediate from the construction of $\mathcal{S}_{v_n}$) and that it preserves geometric realizations (i.e., colimits of simplicial objects). Since a map of $T(n)$-local spectra is an equivalence if and only if it is an equivalence after smashing with some finite type $n$ spectrum, it suffices to show that $\Phi_V$ preserves geometric realizations, for $V$ a finite space of type $n$. This functor can be expressed as the following colimit:
\begin{equation*}
\Phi_V(X) = \varinjlim(\Sigma^\infty\mathrm{Map}_*(V,X) \rightarrow \Sigma^{\infty-d}\mathrm{Map}_*(V,X) \rightarrow \cdots).
\end{equation*}
Therefore it suffices to show that the functor
\begin{equation*}
L_{T(n)}\Sigma^\infty\mathrm{Map}_*(V,-): \mathcal{S}_{v_n} \rightarrow \mathrm{Sp}_{T(n)}
\end{equation*}
preserves geometric realizations. Generally, a functor of the form $\mathrm{Map}_*(V,-)$ only preserves geometric realizations of diagrams of spaces which are at least $\mathrm{dim}(V)$-connected. But this will suffice; we can take the dimension of the space $V_{n+1}$ used to define the localization $L_n^f$ to be at least the dimension of $V$.
\end{proof}

\begin{remark}
Although the proof of Theorem \ref{thm:ehmm} is quite formal, the conclusion is perhaps surprising; it states that the $v_n$-periodic part $\mathcal{S}_{v_n}$ of unstable homotopy theory can be completely described in terms of stable homotopy theory, namely the stable $\infty$-category $\mathrm{Sp}_{v_n}$ and the monad $\Phi_n\Theta_n$.
\end{remark}

It now remains to argue that $\Phi_n\Theta_n$ is in fact the free spectral Lie algebra monad on $\mathrm{Sp}_{v_n}$. This will use some special features of the $\infty$-category of functors from the $\infty$-category of $T(n)$-local spectra to itself. It turns out that in this context the relation between operads and monads is much tighter than in a general symmetric monoidal $\infty$-category. To explain the situation we introduce some terminology:

\index{coanalytic functor}
\begin{definition}
\label{def:coanalytic}
A functor $F: \mathrm{Sp}_{T(n)} \rightarrow \mathrm{Sp}_{T(n)}$ is \emph{coanalytic} if it is equivalent to one of the form $F_{\mathcal{O}}$, with $\mathcal{O}$ a symmetric sequence of $T(n)$-local spectra:
\begin{equation*}
F(X) \simeq L_{T(n)}\bigoplus_{k \geq 1} (\mathcal{O}(k) \otimes X^{\otimes k})_{h\Sigma_k}.
\end{equation*}
We write $\mathrm{coAn}(\mathrm{Sp}_{T(n)})$ for the full subcategory of $\mathrm{Fun}(\mathrm{Sp}_{T(n)},\mathrm{Sp}_{T(n)})$ on the coanalytic functors.
\end{definition}

We discussed the assignment
\begin{equation*}
\mathrm{SymSeq}(\mathcal{C}) \rightarrow \mathrm{Fun}(\mathcal{C},\mathcal{C})\colon \mathcal{O} \mapsto F_{\mathcal{O}}
\end{equation*}
in Section \ref{sec:koszul}. For general $\mathcal{C}$ this is far from fully faithful. However, the $T(n)$-local setting is quite special:

\begin{proposition}
\label{prop:coanalytic}
The functor above gives an equivalence of $\infty$-categories $\mathrm{SymSeq}(\mathrm{Sp}_{T(n)}) \rightarrow \mathrm{coAn}(\mathrm{Sp}_{T(n)})$.
\end{proposition}
\begin{proof}[Sketch of proof]
The proof consists of two ingredients. First, one needs the fact that any natural transformation
\begin{equation*}
(\mathcal{O}(k) \otimes X^{\otimes k})_{h\Sigma_k} \rightarrow (\mathcal{P}(l) \otimes X^{\otimes l})_{h\Sigma_l}
\end{equation*}
between such homogeneous functors is null whenever $k \neq l$. This follows from the general theory of Goodwillie calculus when $k > l$. For the case $k < l$ one needs the additional fact that Tate spectra associated with the symmetric groups vanish in the $T(n)$-local category. This is a fundamental result of Kuhn \cite{kuhntate}; a short alternative proof is provided by Clausen--Mathew \cite{clausenmathew}. The second ingredient is that any natural transformation from a $k$-homogeneous functor
\begin{equation*}
X \mapsto (\mathcal{O}(k) \otimes X^{\otimes k})_{h\Sigma_k}
\end{equation*}
to a coanalytic functor factors through a finite sum of layers. This uses a nilpotence argument of Mathew, which ultimately relies on Tate vanishing again (see the appendix of \cite{heuts}).
\end{proof}

The equivalence of Proposition \ref{prop:coanalytic} sends the composition product of symmetric sequences to the composition of functors. Hence we find the following alternative description of (co)operads in the $T(n)$-local setting:

\index{operad}
\begin{corollary}
\label{cor:operadTn}
The $\infty$-category of operads (resp. of cooperads) in $T(n)$-local spectra is equivalent to the $\infty$-category of monoids (resp. comonoids) in the $\infty$-category $\mathrm{coAn}(\mathrm{Sp}_{T(n)})$. In other words, a (co)monad on $\mathrm{Sp}_{T(n)}$ whose underlying functor is coanalytic corresponds essentially uniquely to a (co)operad in $\mathrm{Sp}_{T(n)}$.
\end{corollary}


The obvious example, which is both an operad and a cooperad, is of course the identity functor of $\mathrm{Sp}_{T(n)}$. The first nontrivial example is the following:

\begin{theorem}[Kuhn \cite{kuhnandrequillen}]
\label{thm:kuhn}
For $E \in \mathrm{Sp}_{T(n)}$ there is a natural equivalence
\begin{equation*}
\Sigma^\infty_{v_n}\Omega^\infty_{v_n}(E) \simeq L_{T(n)}\bigoplus_{k \geq 1} E^{\otimes k}_{h\Sigma_k}.
\end{equation*}
In particular the comonad $\Sigma^\infty_{v_n}\Omega^\infty_{v_n}$ is coanalytic, hence a cooperad. 
\end{theorem}

This result essentially follows from a theorem of Kuhn \cite{kuhnandrequillen} on the splitting of the functor $L_{T(n)}\Sigma^\infty\Omega^\infty$ on a suitable class of spectra. Here we outline a different approach:

\begin{proof}[Sketch of proof]
For a commutative ring spectrum $R$ and a spectrum $X$, there is the well-known adjunction
\begin{equation*}
\mathrm{Map}_{\mathrm{CAlg}}(\Sigma^\infty_+\Omega^\infty X, R) \simeq \mathrm{Map}_{\Omega^\infty}(\Omega^\infty X, \mathrm{GL}_1 R)
\end{equation*}
where the right-hand side denotes the space of infinite loop maps from $\Omega^\infty X$ into the space of units $\mathrm{GL}_1 R$. A rather straightforward adaptation of this setup to our context provides an adjunction
\begin{equation*}
\mathrm{Map}_{\mathrm{CAlg}^{\mathrm{nu}}}(\Sigma^\infty_{v_n}\Omega^\infty_{v_n} E, R) \simeq \mathrm{Map}_{\Omega^\infty}(\Omega^\infty_{v_n} E, M(\mathrm{GL}_1 R))
\end{equation*}
for a non-unital $T(n)$-local commutative ring spectrum $R$ and $E \in \mathrm{Sp}_{T(n)}$. Here $M: \mathcal{S}_* \rightarrow \mathcal{S}_{v_n}$ denotes the localization functor. In fact, $M(\mathrm{GL}_1 R)$ can be identified with $\Omega^\infty_{v_n} R$ as an object of $\mathcal{S}_{v_n}$, but the $\mathbf{E}_\infty$-structure corresponds to the multiplication, rather than addition, on $R$. Applying $\Phi_n$ and using $\Phi_n\Omega^\infty_{v_n} \simeq \mathrm{id}$ then gives a further equivalence
\begin{equation*}
\mathrm{Map}_{\Omega^\infty}(\Omega^\infty_{v_n} E, M(\mathrm{GL}_1 R)) \simeq \mathrm{Map}(E, R),
\end{equation*}
so that we have found a natural equivalence
\begin{equation*}
\mathrm{Map}_{\mathrm{CAlg}^{\mathrm{nu}}}(\Sigma^\infty_{v_n}\Omega^\infty_{v_n} E, R) \simeq \mathrm{Map}(E, R).
\end{equation*}
On the other hand, the universal property of the free non-unital commutative algebra also provides an equivalence
\begin{equation*}
\mathrm{Map}_{\mathrm{CAlg}^{\mathrm{nu}}}(L_{T(n)}\mathrm{Sym}_{\geq 1} E, R) \simeq \mathrm{Map}(E,R)
\end{equation*}
with
\begin{equation*}
L_{T(n)}\mathrm{Sym}_{\geq 1} E = L_{T(n)}\bigoplus_{k \geq 1} E^{\otimes k}_{h\Sigma_k}
\end{equation*}
denoting the spectrum of the theorem. The conclusion now follows from the Yoneda lemma.
\end{proof}

\begin{remark}
Chasing through the proof above gives an explicit description of the equivalence of the theorem. Applying $\Phi_n$ to the unit map $\eta: \Omega^\infty_{v_n} E \rightarrow \Omega^\infty_{v_n}\Sigma^\infty_{v_n}\Omega^\infty_{v_n} E$ and using $\Phi_n \Omega^\infty_{v_n} \simeq \mathrm{id}$ gives a natural map
\begin{equation*}
\lambda: E \rightarrow L_{T(n)}\Sigma^\infty_{v_n}\Omega^\infty_{v_n} E.
\end{equation*}
Since the right-hand side is a non-unital commutative ring spectrum, this map $\lambda$ naturally extends to a map of non-unital commutative rings
\begin{equation*}
L_{T(n)}\mathrm{Sym}_{\geq 1}(E) \rightarrow \Sigma^\infty_{v_n}\Omega^\infty_{v_n} E
\end{equation*}
which is an equivalence by the theorem.
\end{remark}

In fact, as already suggested by the formula of Theorem \ref{thm:kuhn}, the cooperad $\Sigma^\infty_{v_n}\Omega^\infty_{v_n}$ really plays the role of the commutative cooperad: the terms of the corresponding symmetric sequence are (the $T(n)$-localization of) the sphere spectrum in each degree and one can show (e.g. using Goodwillie calculus) that the cooperad structure maps are as expected. We will use Theorem \ref{thm:kuhn} in two ways. The first is a very useful characterization of coanalytic functors:

\index{coanalytic functor}
\begin{proposition}
\label{prop:coanalytic2}
A functor $F\colon \mathrm{Sp}_{T(n)} \rightarrow \mathrm{Sp}_{T(n)}$ is coanalytic if and only if it preserves filtered colimits and geometric realizations.
\end{proposition}

A proof of this result is given in \cite{heuts} following an argument of Lurie; essentially, one writes any functor $F$ preserving filtered colimits and geometric realizations as a colimit of functors closely resembling $\Sigma^\infty_{v_n}\Omega^\infty_{v_n}$ and uses the fact that a colimit of coanalytic functors is coanalytic.

Our goal is to analyze the monad $\Phi_n\Theta_n$. First, one observes it is actually corresponds to an operad (cf. Corollary \ref{cor:operadTn}):

\begin{corollary}
\label{cor:PhiTheta}
The functor $\Phi_n\Theta_n$ is coanalytic.
\end{corollary}
\begin{proof}
By Proposition \ref{prop:coanalytic2} it suffices to check that $\Phi_n\Theta_n$ preserves filtered colimits and geometric realizations. For $\Theta_n$ there is nothing to check since it is a left adjoint; for $\Phi_n$, the fact that it preserves geometric realizations was part of the proof of Theorem \ref{thm:ehmm}. For filtered colimits one may reduce to $\Phi_V$ as usual, where it is obvious. 
\end{proof}

It remains to relate the operad $\Phi_n\Theta_n$ to the spectral Lie operad. In fact, we will indicate how to produce a map of operads
\begin{equation*}
\gamma: \Phi_n\Theta_n \rightarrow \mathrm{Cobar}(\Sigma^\infty_{v_n}\Omega^\infty_{v_n}).
\end{equation*}
The right-hand side is the $T(n)$-local spectral Lie operad, so we should then prove that $\gamma$ is an equivalence. To do this we proceed as follows. As explained in Section \ref{sec:monads}, the adjoint pair $(\Sigma^\infty_{v_n}, \Omega^\infty_{v_n})$ gives a `cosimplicial resolution' of the identity functor of $\mathcal{S}_{v_n}$:
\[
\begin{tikzcd}
\mathrm{id}_{\mathcal{S}_{v_n}} \ar{r} & \Omega^\infty_{v_n}\Sigma^\infty_{v_n}  \ar[shift left]{r} \ar[shift right]{r} & \Omega^\infty_{v_n}\Sigma^\infty_{v_n}\Omega^\infty_{v_n}\Sigma^\infty_{v_n} \ar{l} \ar{r} \ar[shift left = 2]{r} \ar[shift right = 2]{r} & 
 \Omega^\infty_{v_n}(\Sigma^\infty_{v_n}\Omega^\infty_{v_n})^2\Sigma^\infty_{v_n} \cdots \ar[shift left]{l} \ar[shift right]{l}.
\end{tikzcd}
\]
Now precompose with $\Theta_n$, postcompose with $\Phi_n$, and apply the equivalence $\Phi_n\Omega^\infty_{v_n} \simeq \mathrm{id}_{\mathrm{Sp}_{T(n)}} \simeq \Sigma^\infty_{v_n}\Theta_n$ to get the coaugmented cosimplicial object
\[
\begin{tikzcd}
\Phi_n\Theta_n \ar{r} & \mathrm{id}_{\mathrm{Sp}_{T(n)}} \ar[shift left]{r} \ar[shift right]{r} & \Sigma^\infty_{v_n}\Omega^\infty_{v_n} \ar{l} \ar{r} \ar[shift left = 2]{r} \ar[shift right = 2]{r} & (\Sigma^\infty_{v_n}\Omega^\infty_{v_n})^2 \cdots \ar[shift left]{l} \ar[shift right]{l}.
\end{tikzcd}
\]
The totalization produces $\mathrm{Cobar}(\Sigma^\infty_{v_n}\Omega^\infty_{v_n})$, the coaugmentation gives the map $\gamma$. The following is then the final step in the proof of Theorem \ref{thm:vnperiodicspaces}:

\index{spectral Lie operad}
\begin{theorem}
The map $\gamma$ is an equivalence, so that $\Phi_n\Theta_n$ is equivalent to the $T(n)$-localization of the spectral Lie operad $\mathbf{L}$.
\end{theorem}
\begin{proof}[Sketch of proof]
Checking whether a natural transformation between coanalytic functors is an equivalence can be done at the level of Goodwillie derivatives. On the left-hand side, the fact that $\Theta_n$ preserves colimits and $\Phi_n$ preserves limits, as well as filtered colimits, implies that
\begin{equation*}
D_k(\Phi_n\Theta_n) \simeq \Phi_n \circ D_k\mathrm{id}_{\mathcal{S}_{v_n}} \circ \Theta_n.
\end{equation*}
This reduces the verification to checking that
\begin{equation*}
\partial_*\mathrm{id}_{\mathcal{S}_{v_n}} \rightarrow \mathrm{Cobar}(\mathbf{1}, \partial_*(\Sigma^\infty_{v_n}\Omega^\infty_{v_n}), \mathbf{1})
\end{equation*}
is an equivalence, which is yet another version of the result of Arone--Ching (Theorem 0.3 of \cite{aroneching}) already mentioned in Section \ref{sec:Lie}.
\end{proof}

\begin{remark}
Theorem \ref{thm:vnperiodicspaces} is proved in the rather abstract setting of $T(n)$-local homotopy theory. However, one can specialize it to obtain a version for $K(n)$-local homotopy theory as well. 
\end{remark}

\index{Whitehead product}
\begin{remark}
After Theorem \ref{thm:rationalhomotopy} we indicated that the Lie model of a rational space in particular encodes the Whitehead products on rational homotopy groups. Similarly, the spectral Lie algebra model of Theorem \ref{thm:vnperiodicspaces} will in particular encode the Whitehead products on $v_n$-periodic homotopy groups. A concise proof of this can be given by `differentiating' the Hilton--Milnor theorem.
\end{remark}

\index{Koszul duality}
In the remainder of this section we discuss to what extent there is a model for $\mathcal{S}_{v_n}$ in terms of commutative coalgebras, Koszul dual to the Lie algebra model provided by Theorem \ref{thm:vnperiodicspaces}. This is closely related to recent work of Behrens--Rezk \cite{behrensrezk}, as we will explain.

As in the rational case (cf. Section \ref{sec:rationalhomotopy}), one can construct a functor
\begin{equation*}
C_{v_n}: \mathcal{S}_{v_n} \rightarrow \mathrm{coCAlg}^{\mathrm{nu}}(\mathrm{Sp}_{v_n})
\end{equation*}
by observing that each space $X \in \mathcal{S}_{v_n}$ is a non-unital commutative algebra with respect to smash product (using the diagonal as comultiplication) and that the suspension spectrum functor $\Sigma^\infty_{v_n}$ preserves smash products. Alternatively, identifying $\mathcal{S}_{v_n}$ with $\mathrm{Lie}(\mathrm{Sp}_{v_n})$ as in Theorem \ref{thm:vnperiodicspaces}, one can use Theorem \ref{thm:koszul1} to produce a functor from $\mathcal{S}_{v_n}$ to the $\infty$-category of commutative coalgebras in $\mathrm{Sp}_{v_n}$. The functor $C_{v_n}$ admits a right adjoint for which we write $R_{v_n}$. Its existence follows from the adjoint functor theorem; alternatively, it can be constructed much more concretely as the \emph{primitives} of a coalgebra. For this reason we denote the composite functor $\Phi_n R_{v_n}$ by
\begin{equation*}
\mathrm{prim}_{v_n}: \mathrm{coCAlg}^{\mathrm{nu}}(\mathrm{Sp}_{v_n}) \rightarrow \mathrm{Sp}_{v_n}.
\end{equation*}
We remind the reader that these primitives are computed by a cobar construction, which is formally dual to the derived indecomposables (or $\mathrm{TAQ}$) of a non-unital commutative ring spectrum, as discussed in Section \ref{sec:koszul}. The left adjoint of $\mathrm{prim}_{v_n}$ is the functor which equips a spectrum $E \in \mathrm{Sp}_{v_n}$ with the trivial non-unital commutative algebra structure.
\index{topological Andr\'{e}-Quillen homology}

The unit of the adjunction described above is a map
\begin{equation*}
\mathrm{id}_{\mathcal{S}_{v_n}} \rightarrow R_{v_n} C_{v_n},
\end{equation*}
which admits the following descriptions:
\begin{itemize}
\item[(1)] It is the \emph{Goodwillie completion}, in the sense that the functor $R_{v_n} C_{v_n}$ is the limit of the Goodwillie tower of the identity on the $\infty$-category $\mathcal{S}_{v_n}$.
\item[(2)] After identifying $\mathcal{S}_{v_n}$ with the $\infty$-category $\mathrm{Lie}(\mathrm{Sp}_{v_n})$, it is the \emph{pronilpotent completion} of a spectral Lie algebra. This is defined as follows. For every $k \geq 1$ one considers the truncation $t_k\colon \mathbf{L} \rightarrow \tau_{\leq k} \mathbf{L}$, which is an equivalence in arities up to $k$ and has $\tau_{\leq k} \mathbf{L}(j) \cong 0$ for $j > k$. Any $\mathbf{L}$-algebra $X$ then admits a truncation $t^*_k(t_!)_k X$, by pushing forward and pulling back along $t_k$. The inverse limit over $k$ is the pronilpotent completion of $X$. The connection with (1) is that the functor $t^*_k(t_k)_!$ is the $k$-excisive approximation (in the sense of Goodwillie) of the identity functor of the $\infty$-category of spectral Lie algebras. (This perspective originates in \cite{pereira}.)
\end{itemize}
\index{Goodwillie tower of the identity}

It is not difficult to construct examples of objects in $\mathrm{Lie}(\mathrm{Sp}_{v_n})$ which are \emph{not} complete in the above sense (e.g. most free spectral Lie algebras). Consequently, the functor $C_{v_n}$ cannot be an equivalence of $\infty$-categories. Nonetheless, the comparison with coalgebras is a very useful way of `approximating' the $\infty$-category $\mathcal{S}_{v_n}$ and an effective method of computing the value of the Bousfield--Kuhn functor $\Phi$ on many spaces of interest, including spheres. We will conclude this section by making this precise, inspired by the work of Behrens--Rezk \cite{behrensrezk}.

Consider the following diagram of adjoint pairs, with left adjoints on top or on the left:
\[
\begin{tikzcd}
&& \mathcal{S}_* \ar[shift right]{d}[swap]{L_n^f} & \\
\mathrm{Sp}_{T(n)} \ar[shift left]{r}{\Theta_n} & \mathcal{S}_{v_n} \ar[shift left]{l}{\Phi_n} \ar[shift left]{r} & L_n^f\mathcal{S}_* \ar[shift left]{l}{M} \ar[shift left]{r}{C_{L_n^f}} \ar[shift right]{u} & \mathrm{coCAlg}^{\mathrm{nu}}(\mathrm{Sp}_{T(n)}). \ar[shift left]{l}{R_{L_n^f}}
\end{tikzcd}
\]

Here $L_n^f$ is the localization away from a finite suspension space of type $n+1$, as in the previous section. The functor $C_{L_n^f}$ assigns to an $L_n^f$-local space $X$ the $T(n)$-localization of its suspension spectrum $L_{T(n)}\Sigma^\infty X$, with coalgebra structure using the diagonal of $X$ as usual. Recall our slight abuse of notation, writing $\Phi_n$ for the functor $\mathcal{S}_{v_n} \rightarrow \mathrm{Sp}_{T(n)}$ in the diagram, as well as for the composition 
\begin{equation*}
\mathcal{S}_* \xrightarrow{L_n^f} L_n^f\mathcal{S}_* \xrightarrow{M} \mathcal{S}_{v_n} \xrightarrow{\Phi_n} \mathrm{Sp}_{T(n)}.
\end{equation*}
Using the unit of the adjoint pair $(C_{L_n^f}, R_{L_n^f})$ we find for every pointed space $X$ the \emph{comparison map}
\begin{equation*}
\Phi_n(X) \rightarrow \Phi_n R_{L_n^f} C_{L_n^f}(L_n^f X) \cong \mathrm{prim}_{v_n} C_{L_n^f}(L_n^f X).
\end{equation*}
The crucial point is that the right-hand side can be identified with the limit of the Goodwillie tower
\begin{equation*}
\cdots \rightarrow (P_3\Phi_n)(X) \rightarrow (P_2\Phi_n)(X) \rightarrow (P_1\Phi_n)(X) \simeq L_{T(n)}\Sigma^\infty X
\end{equation*}
of $\Phi_n$. A variant of this idea was established in the work of Behrens--Rezk \cite{behrensrezk}. The perspective sketched here is discussed in Section 5 of \cite{heuts}.

\index{$\Phi_n$-good space}
\begin{definition}
\label{def:Phigood}
A pointed space $X$ is \emph{$\Phi_n$-good} if the map
\begin{equation*}
\Phi_n X \rightarrow \varprojlim_k (P_k\Phi_n)(X)
\end{equation*}
is an equivalence. If this map is an equivalence after $K(n)$-localization, then $X$ is \emph{$\Phi_{K(n)}$-good}.
\end{definition}

A consequence of the discussion above is therefore:

\begin{theorem}
\label{thm:BR}
A pointed space $X$ is $\Phi_n$-good if and only if the comparison map
\begin{equation*}
\Phi_n(X) \rightarrow \mathrm{prim}_{v_n} C_{L_n^f}(L_n^f X)
\end{equation*}
is an equivalence.
\end{theorem}

The work of Behrens--Rezk \cite{behrensrezk} is in the $K(n)$-local setting and uses commutative algebras, rather than coalgebras. The dual of the commutative coalgebra $\Sigma^\infty X$ is the commutative non-unital ring spectrum $\mathbb{S}^X$. By taking Spanier--Whitehead duals, the derived primitives of $\Sigma^\infty X$ admit a map
\begin{equation*}
\mathrm{prim}(\Sigma^\infty X) \rightarrow \mathrm{TAQ}(\mathbb{S}^{X})^{\vee}
\end{equation*}
to the dual of the derived indecomposables of $\mathbb{S}^X$. In this way one gets the comparison map of the following result:

\begin{corollary}[Behrens--Rezk \cite{behrensrezk}]
\label{cor:BR}
If $X$ is a pointed space for which $L_{K(n)}\Sigma^\infty X$ is $K(n)$-locally dualizable, then $X$ is $\Phi_{K(n)}$-good if and only if the comparison map
\begin{equation*}
L_{K(n)} \Phi_n(X) \rightarrow L_{K(n)}(\mathrm{TAQ}(\mathbb{S}^{X}))^{\vee}
\end{equation*}
is an equivalence.
\end{corollary}

A discussion of these results from a different perspective is in the chapter of Arone--Ching in this volume. The use of Corollary \ref{cor:BR} is that the codomain of the comparison map is amenable to explicit calculation, using the techniques developed by Behrens--Rezk in \cite{behrensrezk}. Some examples of $\Phi_n$-goodness are the following:
\begin{itemize}
\item[(1)] Spheres are $\Phi_n$-good. This is a key result of Arone--Mahowald \cite{aronemahowald}. They prove something even stronger, namely that the map $\Phi_n(X) \rightarrow (P_k\Phi_n)(X)$ is an equivalence if $X$ is a sphere $S^l$ and $k \geq 2p^n$, or even $k \geq p^n$ if $l$ is odd.
\item[(2)] Spaces of the form $\Theta_n(\mathbb{S}^l)$ are $\Phi_n$-good.
\end{itemize}
There are also many non-examples. The following were essentially first observed in \cite{brantnerheuts}:
\begin{itemize}
\item[(1)] Suppose $V$ is a type $n$ space with a $v_n$ self-map and write $W = \Sigma^2 V$. Then $W$ is \emph{not} $\Phi_n$-good. Briefly, under these conditions one has $W \simeq \Theta_n(\Sigma^\infty_{v_n} W)$. Translating to $\mathrm{Lie}(\mathrm{Sp}_{v_n})$, $W$ corresponds to a \emph{free} spectral Lie algebra. In particular,
\begin{equation*}
\Phi_n(W) \simeq L_{T(n)} \bigoplus_{k \geq 1} (\mathbf{L}(k) \otimes \Sigma^\infty_{v_n} W^{\otimes k})_{h\Sigma_k}.
\end{equation*}
The limit of the Goodwillie tower of $\Phi_n$ evaluated on $W$ will give the direct product, rather than the direct sum.
\item[(2)] A wedge of spheres is not $\Phi_n$-good. Using the Hilton--Milnor theorem, one can describe $\Phi_n(S^a \vee S^b)$ as a direct sum of terms of the form $\Phi_n(S^l)$, for $l$ ranging over an infinite set determined by $a$ and $b$. Again, the limit of the Goodwillie tower will instead be the direct product.
\end{itemize}

There exist more extreme counterexamples of the following kind. Take $X$ to be the cofiber (in $\mathcal{S}_{v_n}$) of a map which is a $T(n)_*$-equivalence, but not a $v_n$-periodic equivalence. Such maps were essentially first described for $n=1$ by Langsetmo--Stanley \cite{langsetmostanley}, by modifying the Adams self-map of a Moore space. The Goodwillie tower of $\Phi_n$ evaluated on $X$ will then vanish identically, although $X$ itself is non-trivial.

\section{Questions}
\label{sec:questions}

In this final section we list several open questions, loosely grouped by subject.

\textbf{(A) Localization and completion.} 

The first few questions we list below are closely related to some of those raised by Behrens--Rezk in their survey \cite{behrensrezksurvey}. In the $\infty$-category $\mathcal{S}_{v_n}$ there are the following analogues of the usual notions of localization and completion:
\begin{itemize}
\item[(a)] Every object $X \in \mathcal{S}_{v_n}$ admits a \emph{$T(n)$-localization}. Note that this is the same as Bousfield localization of objects in $\mathcal{S}_{v_n}$ with respect to the stabilization functor $\Sigma^\infty_{v_n}\colon \mathcal{S}_{v_n} \rightarrow \mathrm{Sp}_{T(n)}$.
\item[(b)] The \emph{$Q_{v_n}$-completion} of an object $X \in \mathcal{S}_{v_n}$ is the limit of the Bousfield--Kan cosimplicial object
\[
\begin{tikzcd}
\Omega^\infty_{v_n}\Sigma^\infty_{v_n} X \ar[shift left]{r} \ar[shift right]{r} & (\Omega^\infty_{v_n}\Sigma^\infty_{v_n})^2 X \ar{l} \ar{r} \ar[shift left = 2]{r} \ar[shift right = 2]{r} & \cdots. \ar[shift right]{l} \ar[shift left]{l}
\end{tikzcd}
\]
\item[(c)] The \emph{Goodwillie completion} of an object $X \in \mathcal{S}_{v_n}$ is the limit of the Goodwillie tower
\begin{equation*}
\cdots \rightarrow P_3\mathrm{id}_{\mathcal{S}_{v_n}}(X) \rightarrow P_2\mathrm{id}_{\mathcal{S}_{v_n}}(X) \rightarrow P_1\mathrm{id}_{\mathcal{S}_{v_n}}(X) = \Omega^\infty_{v_n}\Sigma^\infty_{v_n} X.
\end{equation*}
of the identity of $\mathcal{S}_{v_n}$ evaluated on $X$. Alternatively, identifying $\mathcal{S}_{v_n}$ with $\mathrm{Lie}(\mathrm{Sp}_{v_n})$, the Goodwillie completion is the pronilpotent completion of a spectral Lie algebra.
\end{itemize}

If $X$ is $Q_{v_n}$-complete, or Goodwillie complete, then it is also $T(n)$-local. An argument of Bousfield can be adapted to show that every $H$-space in $\mathcal{S}_{v_n}$ is already $T(n)$-local. An object $X$ is Goodwillie complete if and only if it is $\Phi_n$-good in the sense of Definition \ref{def:Phigood}. 
\begin{itemize}
\item[(A1)] It is not hard to argue that the class of $\Phi_n$-good spaces is closed under finite products. What other general closure properties does this class have?
\item[(A2)] Do the $Q_{v_n}$-completion and the Goodwillie completion agree?
\item[(A3)] Under what conditions does the $T(n)$-localization of $X$ agree with its $Q_{v_n}$-completion or its Goodwillie completion?
\item[(A4)] What is the relation between the spectral sequence associated with the cosimplicial object of (b) (which is a version of the unstable Adams spectral sequence for the $\infty$-category $\mathcal{S}_{v_n}$) and the $v_n$-periodic unstable Adams and Adams--Novikov spectral sequences studied by Bendersky--Curtis--Miller \cite{benderskycurtismiller}, Bendersky \cite{bendersky}, and Davis--Mahowald \cite{davismahowald}?
\end{itemize}

We say a spectral Lie algebra $X$ is \emph{nilpotent} if it is in the essential image of the pullback functor 
\begin{equation*}
t_k^*: \mathrm{Alg}(\tau_{\leq k}\mathbf{L}) \rightarrow \mathrm{Alg}(\mathbf{L})
\end{equation*}
for some $k \geq 1$. With this terminology, the $T(n)$-local spectral Lie algebras corresponding to spheres are nilpotent, by the results of Arone--Mahowald discussed in the previous section. We know that these particular examples are $\Phi_n$-good.

\begin{itemize}
\item[(A5)] Is any nilpotent $X \in \mathrm{Lie}(\mathrm{Sp}_{T(n)})$ complete in the sense of either (b) or (c)?
\end{itemize}

This would provide a large class of examples of $\Phi_n$-good spaces, since the class of nilpotent spectral Lie algebras satisfies various closure properties not obviously shared by the class of $\Phi_n$-good spaces.

\textbf{(B) Exponents.} 

The torsion part of the homotopy groups of $S^n$ has a $p$-exponent. Better yet, Cohen--Moore--Neisendorfer \cite{CMN1,CMN2} prove that the $p^k$-power map of the $H$-space $\Omega_0^{2k+1}S^{2k+1}$ is nullhomotopic. Consequently, the spectrum $\Phi_n(S^{2k+1})$ has the same $p$-exponent. As usual, one deduces corresponding results for even-dimensional spheres from this using the EHP sequence. One can speculate about versions of such results for $v_i$-exponents with $i > 0$. Wang \cite{wang} computes the homotopy groups of $L_{K(2)}\Phi_2(S^3)$ at primes $p \geq 5$ and shows that $v_1^2$ acts trivially on them. However, the situation is more subtle than before, since the element $v_1$ does \emph{not} act trivially on the Morava $E$-theory $E_*(\Phi_2(S^3))$. 

There are similar results for $p$-exponents of Moore spaces $S^k/p^r$, with $k \geq 2$ and $r \geq 1$. If $p$ is odd, then the work of Cohen--Moore--Neisendorfer \cite{CMN3} and Neisendorfer \cite{neisendorfer} shows that $\Omega^2 S^k/p^r$ has null-homotopic $p^{r+1}$-power map. For $p=2$ and $r \geq 2$ there are 2-primary exponent results by Theriault \cite{theriault}. It seems the remaining case $S^k/2$ is still open. These exponent results for Moore spaces give corresponding exponents for the spectra $\Phi_1 S^k/p^r$. A possible generalization would be the following:

\begin{itemize}
\item[(B1)] If $V$ is the suspension of a type $n$ space, does the spectrum $\Phi_n V$ have a $v_i$-exponent for all $i < n$?
\end{itemize}

Here we say that a spectrum $X$ has a $v_i$-exponent if for any finite type $i$ spectrum $W$ with $v_i$ self-map $v$, the smash product map $v \otimes X$ is nilpotent. Under the equivalence between $\mathcal{S}_{v_n} \simeq \mathrm{Lie}(\mathrm{Sp}_{v_n})$, the space $\Sigma V$ corresponds to the free spectral Lie algebra on the spectrum $\Sigma^{\infty+1} V$. This allows for a reformulation of (B1) in terms of the $v_i$-exponents of the underlying spectrum of that free spectral Lie algebra.


\textbf{(C) The Bousfield--Kuhn functor.}

\index{Bousfield--Kuhn functor}
Although $T(n)$- or $K(n)$-homology equivalences of spaces behave quite differently from $v_n$-periodic equivalences, there are still many statements for homology which have counterparts for $v_n$-periodic homotopy groups. One such example is the Whitehead theorem: if a map $f \colon X \rightarrow Y$ of simply-connected pointed spaces induces an isomorphism in $K(n)$-homology for each $n \geq 0$, then it is a weak homotopy equivalence. This is proved by Bousfield in \cite{bousfieldhomologyequiv}; an alternative proof is given by Hopkins--Ravenel in \cite{hopkinsravenel}. There is a version for $v_n$-periodic homotopy groups which states the following:

\begin{theorem}[Barthel--Heuts--Meier, \cite{BHM}]
If a map $f \colon X \rightarrow Y$ of simply-connected finite CW-complexes is a $v_n$-periodic equivalence for every $n \geq 0$, then $f$ is a $p$-local homotopy equivalence.
\end{theorem}

Note that this result includes a finiteness hypothesis on $X$ and $Y$. It is well-known that a finite spectrum with $K(n)_*X = 0$ also has $K(n-1)_*X = 0$. (In fact, another result of Bousfield \cite{bousfieldKnequiv} states that a space $X$ with $K(n)_*X = 0$ also has $K(i)_*X = 0$ for all $0 < i \leq n$, without assuming finiteness of $X$.)  This inspires the following question:

\begin{itemize}
\item[(C1)] For a finite pointed CW-complex $X$, does $\Phi_n X \simeq 0$ imply $\Phi_i X \simeq 0$ for $i < n$?
\end{itemize}

We already discussed the `two adjunctions' diagram
\[
\begin{tikzcd}
\mathrm{Sp}_{v_n} \ar[shift left]{r}{\Theta_n} & \mathcal{S}_{v_n} \ar[shift left]{r}{\Sigma^\infty_{v_n}} \ar[shift left]{l}{\Phi_n} & \mathrm{Sp}_{v_n} \ar[shift left]{l}{\Omega^\infty_{v_n}}
\end{tikzcd}
\]
at the end of Section \ref{sec:periodicity}. The adjunction on the right can be characterized by the universal property of stabilization (in the world of presentable $\infty$-categories); the functor $\Sigma^\infty_{v_n}$ is the initial colimit-preserving functor from $\mathcal{S}_{v_n}$ to a presentable stable $\infty$-category. A positive answer to the following would give a similar universal property of the Bousfield--Kuhn functor:

\index{costabilization}
\begin{itemize}
\item[(C2)] Is the adjoint pair $(\Theta_n,\Phi_n)$ the costabilization of $\mathcal{S}_{v_n}$? In other words, is $\Theta_n$ the terminal colimit-preserving functor from a presentable stable $\infty$-category to $\mathcal{S}_{v_n}$?
\end{itemize}

\textbf{(D) Beyond monochromatic unstable homotopy theory.}

The chromatic approach to homotopy theory involves two aspects: (1) understanding monochromatic layers and (2) assembling those layers to reconstruct a space or spectrum. So far we have only discussed (1). We will now give a brief discussion of (2) and pose some questions.

We have seen that for any $n$ there are comparison functors
\begin{equation*}
\mathcal{S}_{v_n} \rightarrow \mathrm{coCAlg}^{\mathrm{nu}}(\mathrm{Sp}_{T(n)})
\end{equation*}
that, while not fully faithful in general (except for $n=0$), at least behave reasonably well. Integrally, however, the functor
\begin{equation*}
\mathcal{S}_* \rightarrow \mathrm{coCAlg}^{\mathrm{nu}}(\mathrm{Sp})
\end{equation*}
is quite far from being fully faithful. To get a much better approximation, one should replace the right-hand side by the $\infty$-category of coalgebras for the comonad $\Sigma^\infty\Omega^\infty$ (cf. Remark \ref{rmk:SigmaOmega}). The price to pay, though, is that it is in general not so clear what has to be done in order to upgrade a commutative coalgebra spectrum to a $\Sigma^\infty\Omega^\infty$-coalgebra. A first step is to consider the pull-back square (see Proposition 1.9 of \cite{kuhntate})
\[
\begin{tikzcd}
P_2(\Sigma^\infty\Omega^\infty)(X) \ar{d}\ar{r} & (X \otimes X)^{h\Sigma_2} \ar{d} \\
X \ar{r}{\tau_2} & (X \otimes X)^{t\Sigma_2}. 
\end{tikzcd}
\]
If a commutative coalgebra $X$ has a compatible coalgebra structure for $\Sigma^\infty\Omega^\infty$, then the left-hand vertical map has a section. This is equivalent to the existence of a diagonal lift in the square. In other words, the composite of the comultiplication
\begin{equation*}
\delta_2 \colon X \rightarrow (X \otimes X)^{h\Sigma_2}
\end{equation*}
and the canonical map $(X \otimes X)^{h\Sigma_2} \rightarrow (X \otimes X)^{t\Sigma_2}$ is homotopic to the map 
\begin{equation*}
\tau_2 \colon X \rightarrow (X \otimes X)^{t\Sigma_2}.
\end{equation*}
This latter map exists for any spectrum and is called the \emph{Tate diagonal}. One can think of it as the stable (or linear) shadow of the diagonal map of spaces. (See \cite{klein,heutsapprox,nikolausscholze} for much more discussion.) More generally, one can define the notion of a \emph{Tate coalgebra} as in \cite{heuts}. It is first of all a commutative coalgebra, meaning a spectrum $X$ equipped with maps
\begin{equation*}
\delta_k \colon X \rightarrow (X^{\otimes k})^{h\Sigma_k}
\end{equation*}
for $k \geq 2$ and a coherent system of homotopies relating the $\delta_k$ for various $k$. Secondly, these comultiplications $\delta_k$ have to be compatible with certain generalized Tate diagonals
\begin{equation*}
\tau_k \colon X \rightarrow (X^{\otimes k})^{t\Sigma_k}
\end{equation*}
which are constructed inductively. It is proved in \cite{heuts} that there is an equivalence between the $\infty$-category of simply-connected pointed spaces $\mathcal{S}_*^{\geq 2}$ and the $\infty$-category $\mathrm{coAlg}^{\mathrm{Tate}}(\mathrm{Sp})^{\geq 2}$ of simply-connected Tate coalgebras in spectra. 
\index{Tate coalgebra}

In the $\infty$-category of $T(n)$-local spectra, Tate constructions associated with finite groups are contractible \cite{kuhntate}. Hence the theory of Tate coalgebras in $\mathrm{Sp}_{T(n)}$ reduces to that of commutative coalgebras. One can think of the Tate diagonals $\tau_k$ above as determining the `attaching data' between the $\infty$-categories of commutative coalgebras in $\mathrm{Sp}_{T(n)}$ for varying $n$, which assembles them together into the $\infty$-category of Tate coalgebras.

What is less clear is what the Koszul dual side of this picture should be. The $\infty$-category of spectral Lie algebras cannot be a good model for $\mathcal{S}_*$ without $v_n$-periodic localization; these $\infty$-categories have the same stabilization, but the Tate diagonals on $\mathrm{Sp}$ one would associate with the $\infty$-category $\mathrm{Lie}(\mathrm{Sp})$ are zero, as opposed to the usual Tate diagonals arising as the stabilization of the product on $\mathcal{S}_*$. Said (very) informally, the $k$-invariants of the Goodwillie tower of $\mathrm{Lie}(\mathrm{Sp})$ are trivial, whereas they are not for $\mathcal{S}_*$. However, one could hope that it is possible to assemble the $\infty$-categories $\mathrm{Lie}(\mathrm{Sp}_{T(n)})$ for varying $n$ in a more interesting way:

\begin{itemize}
\item[(D1)] Does there exist a good theory of `transchromatic' spectral Lie algebras, related to $\mathcal{S}_*$ (or $L_n^f\mathcal{S}_*$) by an adjoint pair, which after $v_n$-periodic localization reduces to the theory of $T(n)$-local spectral Lie algebras and the Bousfield--Kuhn functor $\Phi_n$ relating it to $\mathcal{S}_{v_n}$?
\item[(D2)] If question (D1) admits a reasonable answer, then what is the relation to Mandell's $p$-adic homotopy theory?
\end{itemize}



%

\bibliographystyle{plain}
\bibliography{biblio}

\end{document}